\documentclass[a4paper,11pt]{amsart}
\usepackage[english]{babel}
\usepackage{mathrsfs}
\usepackage{amsfonts}
\usepackage[centertags]{amsmath}
\usepackage{amsthm}
\usepackage{enumitem}
\usepackage{latexsym,amsmath,amssymb}
\usepackage{graphicx,color}
\usepackage{comment}
\usepackage{tikz}
\usepackage{epstopdf}
\usepackage[normalem]{ulem}

\allowdisplaybreaks

\setenumerate[1]{label=(\roman*), ref=(\roman*)}
\usepackage[all]{xy}

\newtheorem{thm}{Theorem}[section]
\newtheorem{prop}[thm]{Proposition}

\newtheoremstyle{obs}
  {3pt}
  {3pt}
  {}
  {}
  {\bfseries}
  {.}
  {.5em}
  {}
\theoremstyle{obs}
\newtheorem{remark}[thm]{Remark}

\def\qed{\ifvmode\removelastskip\fi
{\unskip\nobreak\hfil\penalty50\hbox{}\nobreak\hfil \hbox{\vrule
height1.2ex width1.2ex}\parfillskip=0pt \finalhyphendemerits=0
\par \smallskip}}

\begin{document}

\title{An extension to the theory of controlled Lagrangians using the Helmholtz conditions}

\author[Marta Farr\'{e} Puiggal\'{i}, Anthony M. Bloch]{ Marta Farr\'{e} Puiggal\'{i}$^1$, Anthony M. Bloch$^1$}

\address{
	$^1$  Department of Mathematics,
	University of Michigan, 530 Church St. Ann Arbor, 48109, 
	Michigan, USA	
}
\email[Anthony M. Bloch]{abloch@umich.edu}
\email[Marta Farr\'e Puiggal{\'\i}]{mfarrepu@umich.edu}

\maketitle

\begin{abstract}
The Helmholtz conditions are necessary and sufficient conditions for a system of second order differential equations to be variational, that is, equivalent to a system of Euler-Lagrange equations for a regular Lagrangian. On the other hand, matching conditions are sufficient conditions for a class of controlled systems to be variational for a Lagrangian function of a prescribed type, known as the controlled Lagrangian. 
Using the Helmholtz conditions we are able to recover the matching conditions from \cite{BLMfirst}. Furthermore we can derive new matching conditions for a particular class of mechanical systems. It turns out that for this class of systems we obtain feedback controls that only depend on the configuration variables. We test this new strategy for the inverted pendulum on a cart and for the inverted pendulum on an incline.

\vspace{3mm}
	
\textbf{Keywords:} inverse problem, Helmholtz conditions, matching conditions, controlled Lagrangians
	
\vspace{3mm}
	
\textbf{2010 Mathematics Subject Classification:} 49N45, 58E30, 70H03, 70Q05
\end{abstract}

\section{Introduction}

Assume we are given a system of Euler-Lagrange equations with an unstable equilibrium that we want to stabilize using feedback controls applied in certain directions. 
The controlled Lagrangians techniques provide a method for solving this problem for  certain mechanical systems, providing explicit control laws \cite{BLMfirst,BCLMsecond,BKZ2015,ortega,gharesifard}. The method consists in choosing, among a class of feedback controls with parameters, ones that transform the controlled system into a Lagrangian form for a specific type of Lagrangian function that depends on the same parameters as the control. Matching conditions are sufficient conditions, depending on the system and the parameters, which ensure this goal is accomplished. Then energy shaping methods can be used to obtain stability.

Trying to put a system of second order differential equations (SODE) into Lagrangian form is the issue of the inverse problem of the calculus of variations, although no specific form is usually imposed for the new Lagrangian, see \cite{Douglas,2008KP,94CSMBP,STP2002}. If a regular Lagrangian function can be found then the system is called variational. The variationality of a SODE can be  characterized through the Helmholtz conditions, which are a set of necessary and sufficient conditions which depend on a matrix $(g_{ij})$ of functions on $TQ$, but not on the Lagrangian function itself. 

In this paper we will use the Helmholtz conditions in order to recover the matching conditions from \cite{BLMfirst} and obtain additional ones for a specific class of systems. The class of allowed controlled Lagrangians will be slightly extended for this purpose.

In this section we will first review aspects of the method of controlled Lagrangians and the matching conditions. 
Then we will introduce the inverse problem and recall the Helmholtz conditions. Finally we show how both sets of conditions are related.

\subsection{Controlled Lagrangians techniques: matching conditions} \label{matching}

Consider a configuration manifold $Q$ with a Lie group $G$ acting freely and properly on $Q$. Consider also a  Lagrangian function $L$ on $TQ$ of mechanical type, that is, of kinetic minus potential energy type. More precisely we have
$$
L(v_q)=\frac{1}{2}g(v_q,v_q)-V(q), \quad v_q\in T_q Q \, ,
$$
for some Riemannian metric $g$ on $Q$ and some function $V:Q\longrightarrow \mathbb{R}$.  We will assume that the Lagrangian $L$ is invariant under the lifted action of $G$ on $TQ$. In this section we will briefly recall how to construct from $L$ a new Lagrangian function $L_{\tau,\sigma,\rho}$ on $TQ$, known as the controlled Lagrangian. 

Consider the vertical spaces with respect to the projection $\pi:Q\longrightarrow Q/G$, that is, the tangent spaces to the orbits of the $G$-action, and the orthogonal complements with respect to the kinetic energy metric $g$, which will be referred to as horizontal spaces. Then for each $v_q \in T_q Q$ we obtain a unique decomposition
$$
v_q=\mbox{Hor} v_q + \mbox{Ver} v_q \, ,
$$
where $\mbox{Ver} v_q \in T_q Orb(q)$ and $\mbox{Hor} v_q \in T_q Orb(q)^{\bot_g}$.

Let $\mathfrak{g}$ denote the Lie algebra of $G$, $\xi_Q$ the infinitesimal generator corresponding to $\xi\in\mathfrak{g}$ and  let $\tau$ be a $\mathfrak{g}$-valued $G$-equivariant horizontal one-form on $Q$.  
We define the $\tau$-horizontal projection and the $\tau$-vertical projection  respectively as
$$
\mbox{Hor}_\tau: v_q\longmapsto \mbox{Hor} v_q - \tau(v_q)_Q(q) \quad \mbox{ and } \quad  \mbox{Ver}_\tau: v_q \longmapsto \mbox{Ver} v_q + \tau(v_q)_Q(q) \quad \mbox{for} \quad v_q\in T_qQ \, .
$$
The freedom in the controlled Lagrangian $L_{\tau,\sigma,\rho}$ comes from the following choices:
\begin{itemize}
\item A new choice of horizontal space, corresponding to a choice of $\tau$, 
\item A change $g\rightarrow g_\sigma$ of the metric acting on $\tau$-horizontal vectors,
\item A change $g\rightarrow g_\rho$ of the metric acting on vertical vectors.
\end{itemize} 
Once these choices have been made, according to \cite[Definition 1.2]{BLMfirst} the controlled Lagrangian can be defined as
\begin{equation*}
L_{\tau,\sigma,\rho}(v_q)=\frac{1}{2}\left(g_\sigma(\mbox{Hor}_\tau v_q, \mbox{Hor}_\tau v_q) + g_\rho(\mbox{Ver}_\tau v_q, \mbox{Ver}_\tau v_q) \right)-V(q) \, .
\end{equation*}

According to \cite[Theorem 2.1]{BLMfirst}, if $g$ and $g_\sigma$ coincide on the horizontal spaces and further  the horizontal and vertical spaces are $g_\sigma$-orthogonal, then the controlled Lagrangian can be rewritten as
\begin{equation} \label{CL-assumptions}
L_{\tau,\sigma,\rho}(v_q) = L(v_q+\tau(v_q)_Q)+\frac{1}{2}g_\sigma(\tau(v_q)_Q,\tau(v_q)_Q)+\frac{1}{2}\varpi(v_q) \, ,
\end{equation}
where $v_q\in T_qQ$ and $\varpi(v_q)=(g_\rho-g)(\mbox{Ver}_\tau(v_q),\mbox{Ver}_\tau(v_q))$.

Assume now that $G$ is Abelian. 
If we take local coordinates $(x^\alpha,\theta^a)$ on $Q$ such that $x^\alpha$ are coordinates on $Q/G$ and $\theta^a$ are coordinates on $G$, then the given mechanical Lagrangian is written as
\begin{equation*}
L(x^\alpha,\dot{x}^\beta,\dot{\theta}^a)=\frac{1}{2}g_{\alpha\beta}\dot{x}^\alpha\dot{x}^\beta+g_{\alpha a}\dot{x}^\alpha\dot{\theta}^a+\frac{1}{2}g_{ab}\dot{\theta}^a\dot{\theta}^b - V(x^\alpha) \, .
\end{equation*}
In these coordinates, under the same assumptions that provide (\ref{CL-assumptions}), the controlled Lagrangian becomes
\begin{eqnarray*}
L_{\tau,\sigma,\rho}\left(x^\alpha,\dot{x}^\beta,\dot{\theta}^a\right)&=& L\left(x^\alpha,\dot{x}^\beta,\dot{\theta}^a+\tau^a_\alpha \dot{x}^\alpha\right) + \frac{1}{2}\sigma_{ab}\tau^a_\alpha \tau^b_\beta \dot{x}^\alpha \dot{x}^\beta  \\
&& + \frac{1}{2}\varpi_{ab}\left( \dot{\theta}^a + g^{ac}g_{\alpha c}\dot{x}^\alpha +\tau^a_\alpha \dot{x}^\alpha \right) \left( \dot{\theta}^b + g^{bd}g_{\beta d}\dot{x}^\beta +\tau^b_\beta \dot{x}^\beta \right) \, ,
\end{eqnarray*}
where $\sigma_{ab},\varpi_{ab}$ are the coefficients of the last two terms in (\ref{CL-assumptions}) and $\tau^a_\alpha$ are the coefficients of $\tau$.

As in \cite{BLMfirst}  we will refer to the particular choice $g_\rho=g$ as the \textbf{special matching assumption} and write the corresponding controlled Lagrangian as $L_{\tau,\sigma}$.
In this case, according to (\ref{CL-assumptions}), the controlled Lagrangian $L_{\tau,\sigma}$ becomes 
\begin{equation*}
L_{\tau,\sigma}(v_q) = L(v_q+\tau(v_q)_Q)+\frac{1}{2}g_\sigma(\tau(v_q)_Q,\tau(v_q)_Q) \, .
\end{equation*}
Now we will compare the solutions of the system of controlled Euler-Lagrange equations
\begin{eqnarray}
\frac{d}{dt}\frac{\partial L}{\partial \dot{x}^\alpha}-\frac{\partial L}{\partial x^\alpha} = 0 \, , & &
\frac{d}{dt}\frac{\partial L}{\partial \dot{\theta}^a} = u_a \, ,  \label{EL1}
\end{eqnarray}
to the solutions of the systems of Euler-Lagrange equations
\begin{eqnarray}
\frac{d}{dt}\frac{\partial L_{\tau,\sigma}}{\partial \dot{x}^\alpha}-\frac{\partial L_{\tau,\sigma}}{\partial x^\alpha} = 0 \, , & & \frac{d}{dt}\frac{\partial L_{\tau,\sigma}}{\partial \dot{\theta}^a} = 0 \, , \label{CELsim} \\
\frac{d}{dt}\frac{\partial L_{\tau,\sigma,\rho}}{\partial \dot{x}^\alpha}-\frac{\partial L_{\tau,\sigma,\rho}}{\partial x^\alpha} = 0 \, , & & \frac{d}{dt}\frac{\partial L_{\tau,\sigma,\rho}}{\partial \dot{\theta}^a} = 0 \, , \label{CELrho}
\end{eqnarray}
where $u_a$ is chosen in such a way that the $\theta^a$ equation in (\ref{EL1}) coincides with the $\theta^a$ equation in (\ref{CELsim}) or (\ref{CELrho}), depending on whether we are making the special matching assumption or not.

Consider the following two sets of assumptions, known as \textbf{matching conditions} and \textbf{simplified matching conditions} respectively:

\textit{Assumption M1}: $\tau^b_\alpha=-\sigma^{ab}g_{\alpha a}\, ,$

\textit{Assumption M2}: $\sigma^{bd}(\sigma_{ad,\alpha}+g_{ad,\alpha})=2g^{bd}g_{ad,\alpha}\, ,$

\textit{Assumption M3}: $\tau^b_{\alpha,\beta}-\tau^b_{\beta,\alpha}=g^{db}g_{ad,\alpha}\tau^a_\beta\, ,$

\textit{Assumption SM1}: $\sigma_{ab}=\sigma g_{ab}$ for a constant $\sigma\, ,$

\textit{Assumption SM2}: $g_{ab}$ is constant , 

\textit{Assumption SM3}: $\tau^b_{\alpha}=-(1/\sigma)g^{ab}g_{\alpha a}\, ,$

\textit{Assumption SM4}: $g_{\alpha a,\delta}=g_{\delta a,\alpha}\, ,$

\noindent where $,\alpha$ denotes partial derivative with respect to $x^\alpha$.

In \cite[Theorem 2.2]{BLMfirst} it is shown that under the matching conditions M1-M3 the solutions of (\ref{EL1}) and (\ref{CELsim}) coincide. In particular, if the simplified matching conditions SM1-SM4 hold, then  M1-M3 also hold and therefore the solutions of (\ref{EL1}) and (\ref{CELsim}) coincide.

If the simplified matching conditions do not hold, then we may relax the special matching assumption $g_\rho=g$ and consider controlled Lagrangians of the form $L_{\tau,\sigma,\rho}$. In this case we can consider the \textbf{generalized matching conditions}, which provide equivalence of (\ref{EL1}) and (\ref{CELrho}), see \cite[Theorem 1.3]{BKZ2015}:

\textit{Assumption GM1}: $\tau^b_\alpha=-\sigma^{ab}g_{\alpha a}\, ,$

\textit{Assumption GM2}: $\sigma^{bd}(\sigma_{ad,\alpha}+g_{ad,\alpha})=2g^{bd}g_{ad,\alpha}\, ,$

\textit{Assumption GM3}: $\varpi_{ab,\alpha}=0\, ,$

\textit{Assumption GM4}: $\tau^b_{\alpha,\delta}-\tau^b_{\delta,\alpha}+\varpi_{ad}\rho^{bd}(\zeta^a_{\alpha,\delta}-\zeta^a_{\delta,\alpha})-\varpi_{ad}\rho^{dc}g_{ce,\delta}\rho^{eb}\zeta^a_\alpha-\rho^{db}g_{ad,\alpha}\tau^a_\delta=0\, ,$

\noindent where $\zeta^a_\alpha=g^{ac}g_{\alpha c}\, .$

So far we have assumed that the original Lagrangian $L$ is invariant under the action of an Abelian Lie group. If we keep this assumption for the kinetic energy part of the Lagrangian but allow for symmetry breaking in the potential energy part then we can add an extra condition to the simplified matching conditions, namely

\textit{Assumption SM5}: $V_{,\alpha a}g^{ad}g_{\beta d}=V_{,\beta a}g^{ad}g_{\alpha d}\, ,$

\noindent which ensures that, if we choose $g_{\rho}=\rho g_{ab}$ for some constant $\rho$, then (\ref{EL1}) is equivalent to 
\begin{eqnarray}
\frac{d}{dt}\frac{\partial L_{\tau,\sigma,\rho,\epsilon}}{\partial \dot{x}^\alpha}-\frac{\partial L_{\tau,\sigma,\rho,\epsilon}}{\partial x^\alpha} = 0 \, , & & \frac{d}{dt}\frac{\partial L_{\tau,\sigma,\rho,\epsilon}}{\partial \dot{\theta}^a} = 0 \, , \label{ELpot} 
\end{eqnarray}
where 
$$
L_{\tau,\sigma,\rho,\epsilon}=L_{\tau,\sigma}+\frac{1}{2}(\rho-1)g_{ab}(\dot{\theta}^a + g^{ac}g_{\alpha c}\dot{x}^\alpha + \tau^a_\alpha \dot{x}^\alpha)(\dot{\theta}^b + g^{bd}g_{\beta d}\dot{x}^\beta + \tau^b_\beta \dot{x}^\beta)-V_{\epsilon}(x^\alpha,\theta^a) \, ,
$$
and $V_\epsilon$ depends on a real parameter $\epsilon$, see \cite[Theorem III.1]{BCLMsecond}.

\begin{remark}
All of the above mentioned matching conditions are sufficient conditions to ensure equivalence of the systems (\ref{EL1}) and (\ref{CELsim}) or (\ref{CELrho}) or (\ref{ELpot}), but they are not enough to guarantee stability of the desired equilibrium. Conditions for obtaining stability are given in \cite{BLMfirst,BCLMsecond}.
\end{remark}

\subsection{The inverse problem of the calculus of variations: Helmholtz conditions}

Assume we are given an explicit system of second order differential equations
\begin{equation} \label{explicit}
\ddot{q}^i=\Gamma^i(q,\dot{q}) \, , \quad i=1,\ldots,n \, .
\end{equation}
The classical multiplier version of the inverse problem of the calculus of variations poses the following question \cite{Douglas}. Is the system (\ref{explicit}) equivalent to a system  of Euler-Lagrange equations for some regular Lagrangian $L(q,\dot{q})$, that is, to a system of the form
\begin{equation} \label{EL}
\frac{d}{dt}\left( \frac{\partial L}{\partial \dot{q}^i}\right) -\frac{\partial L}{\partial q^i} =0 \, ,
\end{equation}
where the matrix $\left(\frac{\partial^2 L}{\partial \dot{q}^i \partial \dot{q}^j}\right)$ is regular?
More precisely, we can ask for the existence of a regular matrix $(g_{ij}(q,\dot{q}))$, the so-called multiplier matrix, such that
\begin{equation} \label{multiplier-problem}
g_{ij}(\ddot{q}^j-\Gamma^j(q,\dot{q}))=\frac{d}{dt}\left( \frac{\partial L}{\partial \dot{q}^i}\right) -\frac{\partial L}{\partial q^i} \, 
\end{equation}
holds for some regular Lagrangian $L$. Notice that the requirement (\ref{multiplier-problem}) coincides with asking for equivalence of solutions of the two systems (\ref{explicit}) and (\ref{EL}). Indeed, if we put (\ref{EL}) in explicit form we get
$$
\ddot{q}^i=g^{ij}\left( -\frac{\partial^2 L}{\partial \dot{q}^{j} \partial q^k }\dot{q}^{k}+\frac{\partial L }{\partial q^{j}} \right) \, ,
$$
where $(g^{ij})$ is the inverse matrix of $\left(\frac{\partial^2 L}{\partial \dot{q}^i \partial \dot{q}^j}\right)$.
Then equivalence of solutions to (\ref{explicit}) and (\ref{EL}) implies $\Gamma^i(q,\dot{q})=g^{ij}\left(- \frac{\partial^2 L}{\partial \dot{q}^{j} \partial q^k }\dot{q}^{k}+\frac{\partial L }{\partial q^{j}} \right) $, that is, (\ref{multiplier-problem})  is satisfied.

Determining if (\ref{multiplier-problem}) holds for some regular Lagrangian is equivalent to determining the existence of solutions to the Helmholtz conditions given by Douglas in terms of multipliers in \cite{Douglas}:
\begin{eqnarray}
&&\det(g_{ij})\not= 0, \qquad g_{ji}=g_{ij}, \qquad
\frac{\partial g_{ij}}{\partial \dot{q}^k}=\frac{\partial g_{ik}}{\partial \dot{q}^j} \, ,  \label{Helmholtz1} \\
&&\Gamma(g_{ij})-\nabla^k_jg_{ik}-\nabla^k_i g_{kj}=0, \qquad
g_{ik}\Phi^k_j=g_{jk}\Phi^k_i, \label{Helmholtz2} 
\end{eqnarray}
where 
\begin{equation*}
\Gamma=\dot{q}^i\frac{\partial}{\partial q^i}+
\Gamma^i(q, \dot{q})\frac{\partial}{\partial \dot{q}^i}\, , \quad 
\nabla^i_j=-\frac{1}{2}\frac{\partial \Gamma^i}{\partial \dot q^j}\, , \quad
\Phi^k_j=\Gamma\left(\frac{\partial \Gamma^k}{\partial \dot q^j}\right)-2\frac{\partial \Gamma^k}{\partial  q^j}-\frac{1}{2}\frac{\partial \Gamma^i}{\partial \dot q^j}\frac{\partial \Gamma^k}{\partial \dot q^i}\, .
\end{equation*}

There is an earlier version of the inverse problem due to Helmholtz which poses the following question \cite{Helmholtz1887}. Given an implicit system of second order differential equations
$$
\Phi_i(q,\dot{q},\ddot{q})=0 \, , \quad i=1,\ldots,n \, ,
$$
determine whether or not it is  possible to find a regular Lagrangian $L(q,\dot{q})$ such that
\begin{equation} \label{IP:exact}
\Phi_i(q,\dot{q},\ddot{q})=\frac{d}{dt}\left( \frac{\partial L}{\partial \dot{q}^i}\right) -\frac{\partial L}{\partial q^i} \, .
\end{equation} 
Necessary and suffient conditions for (\ref{IP:exact}) to hold are
\begin{eqnarray}
\frac{\partial \Phi_i}{\partial \ddot{q}^j}-\frac{\partial \Phi_j}{\partial \ddot{q}^i} &=&0 \, ,  \label{HC1} \\
\frac{\partial \Phi_i}{\partial q^j}-\frac{\partial \Phi_j}{\partial q^i} -\frac{1}{2}\frac{d}{dt}\left( \frac{\partial \Phi_i}{\partial \dot{q}^j}-\frac{\partial \Phi_j}{\partial \dot{q}^i} \right) &=&0 \, , \label{HC2} \\
\frac{\partial \Phi_i}{\partial \dot{q}^j}+\frac{\partial \Phi_j}{\partial \dot{q}^i}-\frac{d}{dt}\left( \frac{\partial \Phi_i}{\partial \ddot{q}^j}+\frac{\partial \Phi_j}{\partial \ddot{q}^i} \right) &=&0 \, , \label{HC3}
\end{eqnarray}
also known as Helmholtz conditions.
Notice that if the system is given in explicit form
$$
\Phi^i=\delta^{ij}\Phi_{j}=\ddot{q}^i-\Gamma^i(q,\dot{q})=0 \, , \quad i=1,\ldots,n \, ,
$$
then we can also try to solve (\ref{IP:exact}). This is the same as requiring that equations (\ref{Helmholtz1})-(\ref{Helmholtz2}) hold with $g=\mbox{Id}$. By doing this, equations (\ref{HC1})-(\ref{HC3})  are recovered from  (\ref{Helmholtz1})-(\ref{Helmholtz2}).

One third possible question is the  following. Assume again that we are given an implicit system of second order differential equations
$$
\Phi_i(q,\dot{q},\ddot{q})=0  \, , \quad i=1,\ldots,n \, ,
$$
but now we only ask for equivalence of solutions, that is, we don't require (\ref{IP:exact}) or (\ref{multiplier-problem}) a priori, (and we don't necessarily know the expression $\ddot{q}^i=\Gamma^i(q,\dot{q})$). Assume $C=\left( \frac{\partial \Phi}{\partial \ddot{q}} \right)$ is regular.

For this problem we also have a set of Helmholtz conditions, to which we will refer as implicit Helmholtz conditions (IHC):
\begin{eqnarray}
\frac{\partial F_i}{\partial \dot{q}^j}&=&\frac{\partial F_j}{\partial \dot{q}^i}  \label{BB} \, , \\
\frac{\partial^2 F_i}{\partial \dot{q}^j\partial q^k}\dot{q}^k+\frac{\partial F_i}{\partial q^j}+\frac{\partial^2 F_i}{\partial \dot{q}^j \partial \dot{q}^k}\ddot{q}^k-\frac{\partial F_j}{\partial q^i}&=&\frac{\partial F_i}{\partial \dot{q}^k}\frac{\partial \Phi_r}{\partial \dot{q}^j}(C^{-1})^{k r} \label{AB} \, ,  \\
\frac{\partial^2 F_i}{\partial q^j\partial q^k}\dot{q}^k+\frac{\partial^2 F_i}{\partial q^j\partial \dot{q}^k}\ddot{q}^k-\frac{\partial F_i}{\partial \dot{q}^k}\frac{\partial \Phi_r}{\partial q^j}(C^{-1})^{k r}&=&\frac{\partial^2 F_j}{\partial q^i\partial q^k}\dot{q}^k+\frac{\partial^2 F_j}{\partial q^i\partial \dot{q}^k}\ddot{q}^k-\frac{\partial F_j}{\partial \dot{q}^k}\frac{\partial \Phi_r}{\partial q^i}(C^{-1})^{k r}  \label{AA} \, .
\end{eqnarray}
Now the unknowns are $F_j(q,\dot{q})$, $j=1,\ldots,n$, instead of $(g_{ij})$, and correspond to the components of the Legendre transformation for the sought Lagrangian function. These equations are given in \cite{BFFM} and are derived in analogous fashion to the ones given in \cite{BFM2015} using Lagrangian submanifolds.

\subsection{The relationship between matching conditions and Helmholtz conditions} \label{intro-relationship}

All of the matching conditions mentioned in Section \ref{matching} give particular solutions of the Helmholtz conditions (\ref{Helmholtz1})-(\ref{Helmholtz2}) or equivalently (\ref{BB})-(\ref{AA}) if we consider the Legendre transformation corresponding to $L_{\tau,\sigma}$, $L_{\tau,\sigma,\rho}$ or $L_{\tau,\sigma,\rho,\epsilon}$.

For example, under the matching conditions M1-M3 we obtain a controlled SODE for which $L_{\tau,\sigma}$ solves the problem (\ref{multiplier-problem}). Therefore the multipliers $g_{ij}=\frac{\partial^2 L_{\tau,\sigma}}{\partial \dot{q}^i \partial \dot{q}^j}$ must satisfy the Helmholtz conditions (\ref{Helmholtz1})-(\ref{Helmholtz2}), and the components of the Legendre transformation $F_i=\frac{\partial L_{\tau,\sigma}}{\partial \dot{q}^i}$ must satisfy the implicit Helmholz conditions (\ref{BB})-(\ref{AA}). The same holds for $L_{\tau,\sigma,\rho}$ and $L_{\tau,\sigma,\rho,\epsilon}$. 

In this paper we will slightly modify the expression of the controlled Lagrangian. We will consider controlled Lagrangians of the form 
$$
\tilde{L}_{\tau,\sigma}=K_{\tau,\sigma}-\tilde{V}_{\tau,\sigma}(x^\alpha,\theta^a)\, , 
$$
where $K_{\tau,\sigma}$ denotes the kinetic energy part of $L_{\tau,\sigma}$, but the potential energy part $\tilde{V}_{\tau,\sigma}$ does not necessarily coincide with the one in the original Lagrangian.

We will take the corresponding Legendre transformation and impose it as a solution of  the implicit Helmholz conditions (\ref{BB})-(\ref{AA}). By doing so, we should recover the matching conditions M1-M3 as particular solutions, but may find new ones due to the freedom in the potential energy part. Now the unknowns are the free parameters that appear both in the controlled Lagrangian and the controlled SODE.

We will follow this approach in Section \ref{sec:ArbDimSpecial}, where we show explicitly how the matching conditions M1-M3 arise from the Helmholtz conditions (\ref{BB})-(\ref{AA}) if we choose $\tilde{L}_{\tau,\sigma}$ as the new Lagrangian. 
The detailed computations are given in Appendix \ref{appendix}.

In Section \ref{sec-new-tau} we use the implicit Helmholtz conditions (\ref{BB})-(\ref{AA}) for the case of one degree of underactuation and $(g_{ab})$ constant  to obtain an additional matching condition, alternative to SM3. In this case we obtain a feedback control which is independent of velocities.

Finally in Section \ref{sec-matching-examples} we will apply the results of Section \ref{sec-new-tau} to obtain a new stabilizing control for the inverted pendulum on a cart. We will further show how to use the implicit Helmholtz conditions (\ref{BB})-(\ref{AA})  to obtain an additional solution for the inverted pendulum on an incline. 
Stability can be achieved in both cases by appropriate choice of the free parameters.

\section{Arbitrary dimension with special matching assumption} \label{sec:ArbDimSpecial}
In this section we will show how the matching conditions M1-M3 arise from the implicit Helmholtz conditions (\ref{BB})-(\ref{AA}) using the special matching assumption, that is, choosing a controlled Lagrangian with $g_\rho=g$. Since we will use the Legendre transformation of the controlled Lagrangian $L_{\tau,\sigma}$, the potential energy of the new Lagrangian will not play any role in satisfying (\ref{BB})-(\ref{AA}).
 
As a starting point consider a given mechanical Lagrangian of the form
\begin{equation}
L(x^\alpha,\dot{x}^\alpha,\dot{\theta}^a)=\frac{1}{2}\left( g_{\alpha\beta}\dot{x}^\alpha\dot{x}^\beta + 2g_{\alpha a}\dot{x}^\alpha\dot{\theta}^a+g_{ab}\dot{\theta}^a\dot{\theta}^b  \right)-V(x^\alpha) \, ,
\end{equation}
with corresponding Euler-Lagrange equations given by
\begin{eqnarray*}
\Phi_\alpha&=& (g_{\alpha\beta,\gamma}-\frac{1}{2}g_{\gamma\beta,\alpha})\dot{x}^\gamma \dot{x}^\beta + (g_{\alpha a , \gamma} - g_{\gamma a,\alpha})\dot{x}^\gamma \dot{\theta}^a \\
&& +g_{\alpha\beta}\ddot{x}^\beta+g_{\alpha a}\ddot{\theta}^a -\frac{1}{2}g_{ab,\alpha}\dot{\theta}^a \dot{\theta}^b + \frac{\partial V}{\partial x^\alpha} =0 \, , \\
 \Phi_a&=& g_{\alpha a,\gamma}\dot{x}^{\gamma}\dot{x}^{\alpha}+g_{ab,\gamma}\dot{x}^{\gamma}\dot{\theta}^b+g_{\alpha a}\ddot{x}^{\alpha}+g_{ab}\ddot{\theta}^b =0 \, .
\end{eqnarray*}
Now consider a controlled Lagrangian with the special matching assumption $g_\rho=g$, that is, 
\begin{eqnarray*}
L_{\tau,\sigma}(x^\alpha,\dot{x}^\alpha,\dot{\theta}^a)&=&L(x^\alpha,\dot{x}^\alpha,\dot{\theta}^a+\tau^a_\alpha\dot{x}^\alpha)+\frac{1}{2}\sigma_{ab}\tau^a_\alpha \tau^b_\beta \dot{x}^\alpha \dot{x}^\beta \, ,
\end{eqnarray*}
and choose feedback controls $u_a$ such that the $\theta^a$-equations for both $L$ and $L_{\tau,\sigma}$ coincide, that is,
\begin{eqnarray} \label{control}
u_a&=&\left( \frac{d}{dt}\left( \frac{\partial L}{\partial \dot{\theta}^a} \right) - \frac{\partial L}{\partial \theta^a} \right)-\left( \frac{d}{dt}\left( \frac{\partial L_{\tau,\sigma}}{\partial \dot{\theta}^a} \right) - \frac{\partial L_{\tau,\sigma}}{\partial \theta^a}  \right) \nonumber \\
&=&  \frac{d}{dt}\left( \frac{\partial L}{\partial \dot{\theta}^a} - \frac{\partial L_{\tau,\sigma}}{\partial \dot{\theta}^a} \right) = \frac{d}{dt}\left(-g_{ab}\tau^b_\beta\dot{x}^\beta\right)=-(g_{ab}\tau^b_\beta)_{,\gamma}\dot{x}^\beta\dot{x}^{\gamma}-g_{ab}\tau^b_\beta \ddot{x}^\beta \, .
\end{eqnarray}
Then the controlled Euler-Lagrange equations (\ref{EL1}) are
\begin{eqnarray}
\tilde{\Phi}_\alpha&=&\Phi_\alpha =0 \, , \label{CEL1}\\
\tilde{\Phi}_a&=& \Phi_a + (g_{ab}\tau^b_\beta)_{,\gamma} \dot{x}^\beta \dot{x}^\gamma + g_{ab}\tau^b_\beta \ddot{x}^\beta=0 \, . \label{CEL2}
\end{eqnarray}
From $\tilde{\Phi}$ we can compute
$$
\tilde{C}=
\left(
\frac{\partial \tilde{\Phi}}{\partial \ddot{q}}
\right)
=
\left(
\begin{array}{cc}
g_{\alpha \beta} & g_{\alpha b} \\
g_{ a \beta}+g_{ad}\tau^d_\beta & g_{ab}
\end{array}
\right) \, ,
$$
which is assumed to be regular. If we introduce the notations $\tilde{W}:=\tilde{C}^{-1}$, $A_{\alpha\beta}=g_{\alpha \beta}-g_{\alpha b}g^{ab}(g_{a \beta} + g_{ad}\tau^d_\beta)$, where $(g^{ab})$ denotes the inverse matrix of  $(g_{ab})$, and also denote the inverse of $(A_{\alpha\beta})$ by $(A^{\alpha\beta})$ then
$$
\tilde{W}=
\left(
\begin{array}{cc}
\tilde{W}^{\alpha \beta} & \tilde{W}^{\alpha b} \\
\tilde{W}^{a \beta} & \tilde{W}^{a b}
\end{array}
\right)
=
\left(
\begin{array}{cc}
A^{\alpha \beta} & -A^{\alpha\gamma}g_{\gamma d} g^{d b} \\
-(g^{ a b}g_{b \gamma}+\tau^a_\gamma)A^{\gamma\beta} & g^{ab}+(g^{ad}g_{d\gamma}+\tau^a_\gamma)A^{\gamma \nu}g_{\nu e}g^{eb}
\end{array}
\right) \, .
$$
Now we will require that the system (\ref{CEL1})-(\ref{CEL2}) be variational using equations (\ref{BB})-(\ref{AA}), and we will impose the solutions given by the components of the Legendre transformation of the controlled Lagrangian $L_{\tau,\sigma}$, that is,
\begin{eqnarray*}
\tilde{F}_\alpha &=& (g_{\alpha \beta}+ g_{\alpha a}\tau^a_\beta +g_{\beta a}\tau^a_\alpha+g_{ab}\tau^a_\alpha\tau^b_\beta+\sigma_{ab}\tau^a_\alpha\tau^b_\beta)\dot{x}^\beta + (g_{\alpha b}+g_{d b}\tau^d_\alpha)\dot{\theta}^b \, ,\\
\tilde{F}_a &=& (g_{\alpha a}+g_{d a}\tau^d_\alpha)\dot{x}^\alpha +g_{ab}\dot{\theta}^b \, .
\end{eqnarray*}

Keeping in mind that $(q^i)=(q^\alpha,q^a)=(x^\alpha,\theta^a)$, we will use subindices $ab$, $a\beta$ or $\alpha\beta$ next to equations (\ref{BB}), (\ref{AB}) and (\ref{AA}) to denote the subset of equations corresponding to $i=a, j=b$ or $i=a, j=\beta$ or $i=\alpha, j=\beta$ respectively.

If we substitute     $\tilde{\Phi}_\alpha$, $\tilde{\Phi}_a$ and the proposed solutions $\tilde{F}_\alpha$, $\tilde{F}_a$ into equations (\ref{BB})-(\ref{AA}) we  get the following:
\begin{itemize}
\item \underline{Equation (\ref{BB})}: vanishes identically for all indices. 
\item \underline{Equation (\ref{AB})}: 
\begin{itemize}
\item[*] $(\ref{AB})_{ab}$ vanishes identically,
\item[*] $(\ref{AB})_{a\beta}$ vanishes identically,
\item[*] $(\ref{AB})_{\alpha b}$ vanishes using M1 and M3 . Alternatively it vanishes if $(g_{ab})$ is constant and the system has one  degree of underactuation, 
\item[*] $(\ref{AB})_{\alpha\beta}$ vanishes using M1, M2 and M3.
\end{itemize}
\item \underline{Equation (\ref{AA})}:
\begin{itemize}
\item[*] $(\ref{AA})_{ab}$ vanishes identically,
\item[*] $(\ref{AA})_{\alpha b}$ vanishes identically,
\item[*] $(\ref{AA})_{\alpha\beta}$ vanishes using M1, M2 and M3. Alternatively it vanishes for systems with one degree of underactuation.
\end{itemize}
\end{itemize}

\bigskip

Detailed computations proving the above statement are given in Appendix \ref{appendix}.

\section{Arbitrary dimension under assumptions SM1, SM2 and with one degree of underactuation} \label{sec-new-tau}

Recall that, as mentioned in Section \ref{intro-relationship}, when solving the Helmholtz conditions we have used the Legendre transformation of the controlled Lagrangian $L_{\tau,\sigma}$. The Helmholtz conditions guarantee the existence of a Lagrangian with the same Legendre transformation as the controlled Lagrangian which we have used, but the potential energy terms need not coincide. Therefore we consider controlled Lagrangians of the form $\tilde{L}_{\tau,\sigma}=K_{\tau,\sigma}-\tilde{V}(x^\alpha,\theta^a)$ with arbitrary $\tilde{V}$.

Since in this section we deal with systems with one degree of underactuation, we will now use the notation $\tau^a_\alpha=\tau^a_1=:\tau^a$, where $a=2,\ldots,n$. We also use a $'$ instead of $_{,1}$ to denote derivative with respect to $x^1=:x$.

\begin{thm} \label{new-tau}
Under assumptions SM1, SM2 and with one degree of underactuation, there is a controlled Lagrangian $\tilde{L}_{\tau,\sigma}$ such that  the Euler-Lagrange equations for $\tilde{L}_{\tau,\sigma}$ are equivalent to the controlled Euler-Lagrange equations (\ref{EL1}) for $L$ if  $\tau^a$ satisfies the ODE system
\begin{equation} \label{ODE-sol}
2\tau^a g_{1 e}g^{ec}g_{1 c}'+2\tau^a g_{1 e}(\tau^e)'-\tau^a g_{11}' + 2g_{11}(\tau^a)' -2g_{1 c}g^{dc}g_{d 1}(\tau^a)' - 2g_{1 c} \tau^c (\tau^a)' =0 \, , 
\end{equation}
for all $a=2,\ldots, n$ . In the particular case when $dim(Q)=2$ we obtain the new solution
\begin{equation} \label{new-tau-2}
\tau(x) = k \sqrt{ g_{11}(x) g_{22} - g_{12}(x)^2} \, ,
\end{equation}
where $k$ is an arbitrary constant. Notice that one degree of underactuation implies that SM4 holds and therefore we are providing an alternative to the solution given by SM3.
\end{thm}

\begin{proof}
From the computations given in Appendix \ref{appendix} (and summarized at the end of Section \ref{sec:ArbDimSpecial})  we can see that Equation (\ref{AA}) vanishes identically since the assumption of one degree of underactuation implies that Equation $(\ref{AA})_{\alpha\beta}$ is void and also that Equation (\ref{AB}) vanishes for indices $ab$, $a\beta$, and $\alpha b$.
Now under assumption SM1, that is $\sigma_{ab}=\sigma g_{ab}$ for some constant $\sigma$,  we compute Equation (\ref{AB}) for indices $\alpha\beta$, which are just $\alpha=\beta=1$. This gives an ODE system as an extra solution, alternative to SM3.

Indeed, the $\dot{\theta}$ components vanish identically using that $(g_{ab})$ is constant and $\alpha=\beta=\eta$, see (\ref{theta-comp}) in Appendix \ref{appendix}. Imposing that the $\dot{x}$ component also vanishes we get 
\begin{eqnarray*}
0&=& -\left( \tau^c - (g_{1 d}\tau^d+\sigma_{ad}\tau^a \tau^d)A^{11}g_{1 e}g^{ec} \right) \left( g_{1 c}'+g_{1 c}' + (g_{cd}\tau^d)' + (g_{cd}\tau^d)' \right)\dot{x}  \\
&& -\left( \delta_{1}^{1}+(g_{1 d}\tau^d+\sigma_{ad}\tau^a\tau^d)A^{11} \right)  (g_{11}' + g_{11}'-g_{11}')\dot{x}  \\
&& + (g_{1 1} + g_{1 d}\tau^d + g_{1 d}\tau^d +g_{de}\tau^d \tau^e +\sigma_{ed}\tau^e \tau^d  )'\dot{x} \\
&=& -\left( \tau^c - (g_{1 d}\tau^d+\sigma_{ad}\tau^a \tau^d)A^{11}g_{1 e}g^{ec} \right) \left( 2g_{1 c}' + 2(g_{cd}\tau^d)'  \right)\dot{x}  \\
&& - g_{11}'\dot{x}-(g_{1 d}\tau^d+\sigma_{ad}\tau^a\tau^d)A^{11} g_{11}' \dot{x}  + g_{11}'\dot{x} + (2 g_{1 d}\tau^d  +g_{de}\tau^d \tau^e +\sigma_{ed}\tau^e \tau^d  )'\dot{x} \\
&=& - \tau^c \left( 2g_{1 c}' + 2(g_{cd}\tau^d)'  \right)\dot{x}
 + (g_{1 d}\tau^d+\sigma_{ad}\tau^a \tau^d)A^{11}g_{1 e}g^{ec}  \left( 2g_{1 c}' + 2(g_{cd}\tau^d)'  \right) \dot{x}  \\
&& -(g_{1 d}\tau^d+\sigma_{ad}\tau^a\tau^d)A^{11} g_{11}' \dot{x}   + (2 g_{1d}'\tau^d+2 g_{1 d}(\tau^d)'  +g_{de}'\tau^d \tau^e + 2g_{de}(\tau^d)' \tau^e + 2\sigma_{ed}(\tau^e)' \tau^d  )\dot{x} \\
&=& A^{11} \left( (g_{1 d}+\sigma_{ad}\tau^a)\tau^d \left(
g_{1 e}g^{ec}  \left( 2g_{1c}' + 2g_{cd}(\tau^d)'  \right) - g_{11}' \right) + 2A_{11}( g_{1 d}(\tau^d)' + \sigma_{da}(\tau^d)' \tau^a ) \right) \\
&=& A^{11}  (g_{1 d}+\sigma_{ad}\tau^a)\left( \tau^d \left( g_{1 e}g^{ec}  \left( 2g_{1c}' + 2g_{cd}(\tau^d)'  \right) - g_{11}' \right) +2 (\tau^d)' A_{11}\right) \\
&=& A^{11} (g_{1 d}+\sigma_{ad}\tau^a)
\left( 2\tau^d g_{1 e}g^{ec}g_{1c}'+2\tau^d g_{1 e}(\tau^e)'-\tau^d g_{11}' + 2g_{11}(\tau^d)' -2g_{1 c}g^{ec}g_{e 1}(\tau^d)' - 2g_{1 c} \tau^c (\tau^d)' \right) \, .
\end{eqnarray*}

Therefore we have the two solutions M1 and
\begin{equation*} 
2\tau^a g_{1 e}g^{ec}g_{1c}'+2\tau^a g_{1 e}(\tau^e)'-\tau^a g_{11}' + 2g_{11}(\tau^a)' -2g_{1 c}g^{dc}g_{d 1}(\tau^a)' - 2g_{1 c} \tau^c (\tau^a)' = 0 \, ,
\end{equation*}
for each $a=2,\ldots,n$.
Notice that in the case when $\mbox{dim}(Q)=2$ the system (\ref{ODE-sol}) becomes
\begin{equation*}
g^{22}(2g_{12}g_{12}'\tau-g_{11}'g_{22}\tau+2g_{11}g_{22}\tau'-2g_{12}^2\tau') = 0 \, ,
\end{equation*}
and the solution is given by 
\begin{equation*} 
\tau(x) = k \sqrt{ g_{11}(x) g_{22} - g_{12}(x)^2} \, ,
\end{equation*}
where $k$ is an arbitrary constant.
\end{proof}

\begin{prop} \label{prop-u}
Under the assumptions of Theorem \ref{new-tau} and using the new solution given by (\ref{ODE-sol}) we have that the control (\ref{control}) is independent of velocities.
\end{prop}
\begin{proof}
Indeed we have that the equations $\tilde{\Phi}_1=0$ and $\tilde{\Phi}_a=0$ are given by
\begin{eqnarray*}
g_{11}\ddot{x} + g_{1a}\ddot{\theta}^a &=& -\frac{1}{2}g_{11}' \dot{x}^2 -V' \, , \\
(g_{1a}+g_{ab}\tau^b)\ddot{x} + g_{ab}\ddot{\theta}^b &=& -(g_{1a}'+g_{ab}(\tau^b)')\dot{x}^2 \, .
\end{eqnarray*}
Therefore, since $\tilde{C}=\partial \tilde \Phi / \partial \ddot{q}$ is regular, we have 
\begin{equation}
\ddot{x}=A^{11}\left( -\frac{1}{2}g_{11}'\dot{x}^2 +g_{1d}g^{de}g_{1e}'\dot{x}^2 + g_{1d}(\tau^d)'\dot{x}^2 - V' \right)
\end{equation}
and the control (\ref{control}) becomes
\begin{equation}\label{control-position}
u_a= -g_{ab}(\tau^b)'\dot{x}^2-g_{ab}\tau^b A^{11}\left( -\frac{1}{2}g_{11}'\dot{x}^2 +g_{1d}g^{de}g_{1e}'\dot{x}^2 + g_{1d}(\tau^d)'\dot{x}^2 - V' \right) =g_{ab}\tau^b A^{11}V' \, ,
\end{equation}
where in the last equality we have used that $A_{11}=g_{11}-g_{1f}g^{ef}(g_{e1}+g_{ed}\tau^d)$ is nonvanishing in order to get
\begin{eqnarray*}
&&-A_{11}g_{ab}(\tau^b)'-g_{ab}\tau^b\left( -\frac{1}{2}g_{11}' + g_{1d}g^{de}g_{1e}' + g_{1d}(\tau^d)' \right) \\
&=& -g_{11}g_{ab}(\tau^b)'+g_{1f}g^{ef}(g_{e1}+g_{ed}\tau^d)g_{ab}(\tau^b)'+\frac{1}{2}g_{11}'g_{ab}\tau^b - g_{1d}g^{de}g_{1e}'g_{ab}\tau^b - g_{1d}(\tau^d)'g_{ab}\tau^b \\
&=&g_{ab} \left( -g_{11}(\tau^b)' + g_{1f}g^{ef}g_{e1}(\tau^b)' + g_{1d}\tau^d(\tau^b)' +\frac{1}{2}g_{11}'\tau^b  -\tau^b g_{1d}g^{de}g_{1e}' -g_{1d}\tau^b(\tau^d)' \right) \\
&\stackrel{(\ref{ODE-sol})}{=}&0 \, .
\end{eqnarray*}
\end{proof}

\begin{remark}
After substitution of $u_a$ given in Proposition \ref{prop-u} into the system $\tilde{\Phi}_1=0$, $\tilde{\Phi}_a=0$ we get
$$
\left(
\begin{array}{c}
\ddot{x} \\
\ddot{\theta}^a
\end{array}
\right)
=
C^{-1}
\left(
\begin{array}{c}
-\frac{1}{2}g_{11}'\dot{x}^2 - V' \\
-g_{1a}'\dot{x}^2 + g_{ab}\tau^b A^{11}V'
\end{array}
\right) \, .
$$
In dimension $2$ this fits into the class of systems considered in \cite{FM}.
\end{remark}

\section{Examples} \label{sec-matching-examples}

In this section we will see two examples in which alternative $\tau$ solutions can be found to the simplified matching condition SM3. 
The first example fits into Section \ref{sec-new-tau}. We illustrate the new $\tau$ solution for the inverted pendulum on a cart as well as the control given by (\ref{control-position}), which is stabilizing in this case. 
In the second example, which involves a controlled Lagrangian of type $L_{\tau,\sigma,\rho,\epsilon}$, the Helmholtz condition $(\ref{AB})_{\alpha\beta}$ provides an additional solution for $\tau$ as in Theorem \ref{new-tau}. Furthermore, the Helmholtz condition $(\ref{AA})_{\alpha b}$, which vanishes identically in Example \ref{example-inverted-pendulum}, provides now a PDE for $V_\epsilon$, which is more general than the one provided in \cite{BCLMsecond}.

\subsection{Inverted pendulum on a cart} \label{example-inverted-pendulum}

We provide a new stabilizing control for the inverted pendulum on a cart using the solution provided by Theorem \ref{new-tau}.
The system consists of a pendulum of length $l$ and a bob of mass $m$ attached to the top of a cart of mass $M$. The configuration manifold of the system is $Q=S^1\times \mathbb{R}$ with coordinates $(x,s)$ which denote the angle of the pendulum with respect to a vertical line and the position of the cart respectively, as shown in the picture below. The upright position of the pendulum corresponds to $x=0$.

\begin{center}
\begin{tikzpicture}
\draw (0,0) -- (8,0);
\draw (1,-1) -- (1,4);
\draw [->] (1,-0.5) -- (4,-0.5);
\node at (2.5,-0.7) {$s$};
\draw [dashed] (4,-1) -- (4,4);
\draw (2.7,0.3) rectangle (5.3,1.6);
\draw [thick] (4,1.6) -- (5,3.7);
\draw [fill] (5,3.7) circle [radius=0.2];
\draw [fill] (3.35,0.3) circle [radius=0.3];
\draw [fill] (4.65,0.3) circle [radius=0.3];
\draw [<-,domain=70:87] plot ({4+1.3*cos(\x)}, {1.6+1.3*sin(\x)});
\node at (4.3,3.3) {$x$};
\draw [->] (1.5,0.95) -- (2.5,0.95);
\node at (2,1.1) {$u$};
\draw [<->, dotted] (4.3,1.45714) -- (5.3,3.55714);
\node at (5.1,2.5) {$l$};
\end{tikzpicture}
\end{center}

The Lagrangian is given by
$$
L=\frac{1}{2} \left(\alpha  \dot{x}^2+2 \beta  \dot{s} \dot{x} \cos (x)+\gamma  \dot{s}^2\right)+d \cos (x) \, , 
$$
where $\alpha=ml^2,\beta=ml,\gamma=m+M$ and $d=-mgl$ are constants.

If we choose the solution provided by (\ref{new-tau-2}), that is,
$$
\tau(x) = k \sqrt{\alpha\gamma - \beta^2\cos^2(x)} \, ,
$$ 
then from (\ref{control-position}) we obtain the control
\begin{equation} \label{control-inverted-pendulum}
u=g_{22}\tau A^{11}V'=-\frac{d \gamma ^2 k \sin (x) \sqrt{\alpha  \gamma -\beta ^2 \cos ^2(x)}}{\beta  \gamma  k \cos (x) \sqrt{\alpha  \gamma -\beta ^2 \cos ^2(x)}-\alpha  \gamma +\beta ^2 \cos ^2(x)} \, .
\end{equation}

We will now check the stability of the upright position of the pendulum with this control.
To this end we will use the energy function corresponding to the new Lagrangian $\tilde{L}_{\tau,\sigma}$ (with the same Legendre transformation as $L_{\tau,\sigma}$ but a possibly different potential energy term, as remarked above).

When written in explicit form, the controlled Euler-Lagrange equations become
\begin{eqnarray} \label{sode-pendulum-cart}
\ddot{x}&=&
\frac{\sin (x) \left(\frac{d \gamma  \left(\beta ^2 \cos ^2(x)-\alpha  \gamma \right)}{-\beta  \gamma  k \cos (x) \sqrt{\alpha  \gamma -\beta ^2 \cos ^2(x)}+\alpha  \gamma -\beta ^2 \cos
   ^2(x)}-\beta ^2 \dot{x}^2 \cos (x)\right)}{\alpha  \gamma -\beta ^2 \cos ^2(x)} =:F \, , \\
\ddot{s}&=&\sin (x) \left(-\frac{\alpha  d \gamma ^2 k}{\sqrt{\alpha  \gamma -\beta ^2 \cos ^2(x)} \left(\beta  \gamma  k \cos (x) \sqrt{\alpha  \gamma -\beta ^2 \cos ^2(x)}-\alpha  \gamma +\beta ^2 \cos
   ^2(x)\right)}
   \right. \nonumber \\
 &&  
   \left.
   +\frac{\beta  d \cos (x)}{\alpha  \gamma -\beta ^2 \cos ^2(x)}+\frac{\alpha  \beta  \dot{x}^2}{\alpha  \gamma -\beta ^2 \cos ^2(x)}\right) =:G \label{sode-pendulum-cart2} \, .
\end{eqnarray}
We can write the new Lagrangian as 
$$
\tilde{L}_{\tau,\sigma}=\frac{1}{2}\left(\tilde{g}_{11}(x)\dot{x}^2+2\tilde{g}_{12}(x)\dot{x}\dot{s}+\tilde{g}_{22}\dot{s}^2 \right)-\tilde{V}(x,s) \, ,
$$
where
\begin{eqnarray*}
\tilde{g}_{11}(x)&=& \gamma  k^2 (\sigma +1) \left(\alpha  \gamma -\beta ^2 \cos ^2(x)\right)+2 \beta  k \cos (x) \sqrt{\alpha  \gamma -\beta ^2 \cos^2(x)}+\alpha \, , \\
\tilde{g}_{12}(x)&=& \gamma  k \sqrt{\alpha  \gamma -\beta ^2 \cos ^2(x)}+\beta  \cos (x) \, ,\\
\tilde{g}_{22}&=&\gamma \, .
\end{eqnarray*}
Then the equivalence conditions (\ref{multiplier-problem}) are
\begin{eqnarray*}
-\tilde{g}_{11}F-\tilde{g}_{12}G&=&\frac{\partial \tilde{g}_{11}}{\partial x }\dot{x}^2+\frac{\partial \tilde{g}_{12}}{\partial x }\dot{s}\dot{x}- \left(\frac{1}{2}\frac{\partial \tilde{g}_{11}}{\partial x }\dot{x}^2+\frac{\partial \tilde{g}_{12}}{\partial x }\dot{s}\dot{x}+\frac{\partial \tilde{V}}{\partial x} \right) \, , \\
-\tilde{g}_{21}F-\tilde{g}_{22}G&=&\frac{\partial \tilde{g}_{21}}{\partial x }\dot{x}^2-\frac{\partial \tilde{V}}{\partial s} \, ,
\end{eqnarray*}
from which  we get
\begin{eqnarray*}
\frac{\partial \tilde{V}}{\partial x} &=&-\frac{d \left(\gamma ^2 k^2 \sigma +1\right) \sin (x) \left(\alpha  \gamma -\beta ^2 \cos ^2(x)\right)}{\beta  \gamma  k \cos (x)
   \sqrt{\alpha  \gamma -\beta ^2 \cos ^2(x)}-\alpha  \gamma +\beta ^2 \cos ^2(x)} \, , \\
\frac{\partial \tilde{V}}{\partial s} &=& 0 \, .
\end{eqnarray*}
Now we impose conditions such that the new multiplier matrix $(\tilde{g}_{ij})$ will be positive-definite. 
If we introduce the notation 
$$
D=g_{11}g_{22}-g_{12}^2 \quad \mbox{ and } \quad \tilde{D}=\tilde{g}_{11}\tilde{g}_{22}-\tilde{g}_{12}^2 \, ,
$$
then we have $\tilde{D}=D+\sigma(g_{22}\tau)^2$, and therefore we need to choose $\sigma>\frac{-D}{(g_{22}\tau)^2}=\frac{-1}{\gamma^2 k^2}$. We also need $\tilde{g}_{11}>0$, for which it is enough to take $\tau>0$ and $g_{22}(1+\sigma)\tau+2g_{12}>0$. Therefore it is enough to choose $\sigma>0$ and $k>0$.

On the other hand, looking at $\partial \tilde{V} / \partial x$ , notice that we have
$$
d<0, \quad \alpha\gamma -\beta ^2 \cos ^2(x) > 0 \quad \mbox{ and } \quad \gamma ^2 k^2 \sigma +1 >0
$$ 
from the previous choice. Then, in order to get a positive-definite potential energy, we need to impose  
$$\beta  \gamma  k \cos (x)
   \sqrt{\alpha  \gamma -\beta ^2 \cos ^2(x)}-\alpha  \gamma +\beta ^2 \cos ^2(x)>0 \, ,
$$
which, taking $x\in \left( -\frac{\pi}{2},\frac{\pi}{2} \right)$,  reduces to
\begin{equation} \label{stability-cart}
k>\frac{\alpha\gamma-\beta^2\cos^2(x)}{\beta\gamma\cos(x)\sqrt{\alpha\gamma-\beta^2\cos^2(x)}} \, ,
\end{equation}
(but stability of $x=0$ is guaranteed with $k>\frac{\alpha\gamma-\beta^2}{\beta\gamma\sqrt{\alpha\gamma-\beta^2}}$).

Summing up, we can choose positive $\sigma$ and $k$ to guarantee that the new kinetic energy is positive-definite and we can further adjust the constant $k$ in the control to guarantee that the potential energy is positive-definite. Then the energy is a Lyapunov function for (\ref{sode-pendulum-cart})-(\ref{sode-pendulum-cart2}). Notice that the requirement (\ref{stability-cart}) corresponds to $A_{11}<0$.

We now fix the parameters of the system to be $m=0.14 \mbox{ kg}$, $M=0.44 \mbox{ kg}$ and $l=0.215 \mbox{ m}$ as in \cite{BLMfirst} and take the initial conditions to be $\phi(0)=\pi/2-0.2 \mbox{ rad}$, $\dot{\phi}(0)=0.1 \mbox{ rad/s}$, $s(0)=0 \mbox{ m}$, and $\dot{s}(0)=-3 \mbox{ m/s}$, also as in \cite{BLMfirst}.
Below there is a \textsc{matlab}  simulation of this situation with $k=35$:

\includegraphics[height=8cm]{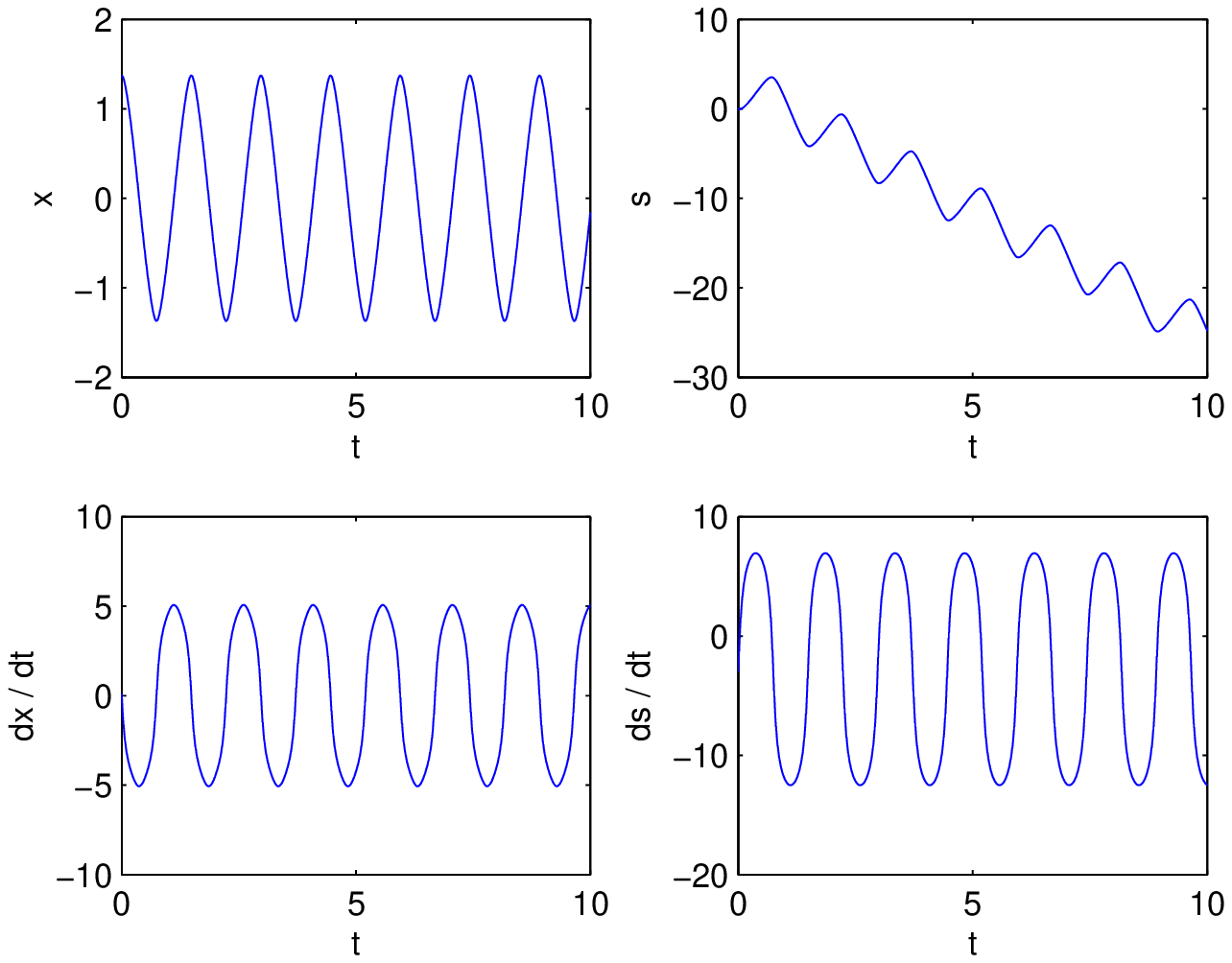}

\begin{remark}
The system (\ref{sode-pendulum-cart})-(\ref{sode-pendulum-cart2}) fits into the class of systems dealt with in \cite{FM} and 
belongs to Case IIa1 from Douglas' classification since
$$
\Phi^2_2=d\gamma\left( 2(\beta^2-\alpha\gamma)\cos(x)\sqrt{\alpha\gamma-\beta^2\cos^2(x)}-2\beta\gamma k (\alpha\gamma+\beta^2\cos^2(x)) \right) \not=0 \, .
$$
This is the same case as the controlled systems that appear in \cite{BLMfirst} and \cite{FM} for the example of the inverted pendulum on a cart.
\end{remark}

\subsection{Inverted pendulum on a cart on an incline} \label{ex:incline}

Consider now the inverted pendulum on a cart, moving on an incline, and denote by $\psi$ the angle between the incline and the horizontal. 
The Lagrangian is given by
\begin{equation*}
L(x,s,\dot{x},\dot{s})=\frac{1}{2}(\alpha\dot{x}^2+2\beta\cos(x-\psi)\dot{x}\dot{s}+\gamma\dot{s}^2) + d\cos(x) +\gamma g \sin(\psi)s \, ,
\end{equation*}
where the potential energy function is $V(x,s)=-d\cos(x) - \gamma g \sin(\psi)s$.

We will assume $g_\rho=\rho g$ for a scalar constant $\rho$, as in \cite{BCLMsecond}, and we will consider the  controlled Lagrangian
\begin{eqnarray*}
L_{\tau,\sigma,\rho,\epsilon} &=& \frac{1}{2}(\alpha\dot{x}^2+2\beta\cos(x-\psi)\dot{x}(\dot{s}+\tau\dot{x})+\gamma(\dot{s}+\tau\dot{x})^2) + d\cos(x) +\gamma g \sin(\psi)s \\
&& +\frac{1}{2}\sigma\gamma\tau^2\dot{x}^2 +\frac{1}{2}(\rho-1)\gamma\left( \dot{s} +\frac{\beta}{\gamma}\cos(x-\psi)\dot{x} + \tau\dot{x} \right)^2 - V_{\epsilon}(x,s) \, .
\end{eqnarray*}

If we require that the $s$ equations for both Lagrangians $L$ and $L_{\tau,\sigma,\rho,\epsilon}$ coincide then we obtain the following control, see \cite[Equation (14)]{BCLMsecond}:
\begin{eqnarray} \label{control-ipci}
u &=& \frac{\rho-1}{\rho}\frac{\partial V}{\partial s} - \frac{\partial V_{\epsilon}}{\partial s} - \frac{d}{dt}(\gamma\tau\dot{x})  \nonumber \\
&=& \frac{1-\rho}{\rho}\gamma g \sin(\psi) - \frac{1}{\rho}\frac{\partial V_{\epsilon}}{\partial s} - \gamma\tau\ddot{x} -\gamma\tau'\dot{x}^2 \, .
\end{eqnarray}
As in Section \ref{sec:ArbDimSpecial} we require now that the system of controlled Euler-Lagrange equations (\ref{EL1}), which in this case becomes
\begin{eqnarray} 
\alpha  \ddot{x}+\beta  \ddot{s} \cos (\psi -x)+ d\sin (x)&=&0 \, , \label{SODEipci-x} \\ 
\dot{x}^2\left(\beta\sin(\psi-x)+\gamma\tau'\right)
+(\beta\cos(\psi-x)+\gamma\tau)\ddot{x} +\gamma\ddot{s} 
+\frac{1}{\rho}\frac{\partial V_{\epsilon}}{\partial s} -\frac{g\gamma}{\rho}\sin(\psi )&=&0 \, , \label{SODEipci-s}
\end{eqnarray}
be variational for a Lagrangian function of the type $L_{\tau,\sigma,\rho,\epsilon}$. Notice that in this case the potential energy term is already free, in contrast to the case $L_{\tau,\sigma,\rho}$ but in similar fashion
 to $\tilde{L}_{\tau,\sigma,\rho}$.

If we write the Helmholtz conditions (\ref{BB})-(\ref{AA}) for the controlled SODE (\ref{SODEipci-x})-(\ref{SODEipci-s}) and impose as a solution the components of the Legendre transformation corresponding to $L_{\tau,\sigma,\rho,\epsilon}$, that is,
\begin{eqnarray*}
F_1&=& \left(\alpha + \frac{\beta^2(\rho-1)\cos^2(\psi -x)}{\gamma}+2\beta\rho\tau\cos(\psi -x)+\gamma(\rho +\sigma )\tau^2\right)\dot{x} \\
&& + \rho(\beta\cos(\psi -x)+\gamma\tau)\dot{s} \, , \\
F_2&=& \rho(\beta \cos(\psi -x)+\gamma\tau)\dot{x}+ \rho\gamma \dot{s} \, ,
\end{eqnarray*}
then we obtain the following:
\begin{itemize}
\item \underline{Equation (\ref{BB})}: vanishes identically for all indices. 
\item \underline{Equation (\ref{AB})}: 
\begin{itemize}
\item[*] $(\ref{AB})_{ab}$ vanishes identically,
\item[*] $(\ref{AB})_{a\beta}$ vanishes identically,
\item[*] $(\ref{AB})_{\alpha b}$ vanishes identically,
\item[*] $(\ref{AB})_{\alpha\beta}$ becomes the following equation:
\begin{equation*}
\frac{\dot{x}\left( \beta\cos(\psi-x)+\gamma\sigma\tau \right) \left( -\beta^2 \sin(2(\psi-x))\tau+(\beta^2-2\alpha\gamma+\beta^2\cos(2(\psi-x))\tau') \right)}{-\alpha\gamma + \beta^2\cos^2(\psi-x) + \beta\gamma\cos(\psi-x)\tau}=0 \, .
\end{equation*}
Therefore we obtain the two solutions
\begin{eqnarray}
\tau&=&-\frac{\beta}{\gamma\sigma}\cos(x-\psi) \, , \label{tau:old:ipci}\\
\tau&=&k\sqrt{\alpha\gamma-\beta^2\cos^2(x-\psi)} \, , \label{tau:new:ipci} 
\end{eqnarray}
where $k$ is a constant. 
The first one corresponds to the simplified matching condition SM3, while  
 the second one coincides with (\ref{new-tau-2}).
\end{itemize}
\item \underline{Equation (\ref{AA})}:
\begin{itemize}
\item[*] $(\ref{AA})_{\alpha b}$ becomes the following second order linear PDE for $V_{\epsilon}(x,s)$:
\begin{equation} \label{PDE-potencial}
A(x)\frac{\partial^2 V_{\epsilon}}{\partial s^2} + \frac{\partial^2 V_{\epsilon}}{\partial s \partial x} = 0 \, , 
\end{equation}
where 
\begin{eqnarray*}
A(x) &=& \frac{\frac{1}{2}\beta(\rho -1)\cos(\psi -x) \left(-2 \alpha\gamma +\beta ^2 \cos(2 (\psi -x))+\beta ^2\right)}{\gamma\rho\left(\alpha\gamma-\beta^2 \cos^2(\psi-x)-\beta\gamma  \tau \cos(\psi -x)\right)} \\
	&& + \frac{\beta\gamma^2 (\rho+\sigma)\tau^2 \cos(\psi -x)+\gamma\rho\tau \left(2 \beta^2 \cos^2(\psi -x)-\alpha \gamma\right)}{\gamma\rho\left(\alpha\gamma -\beta^2 \cos^2(\psi-x)-\beta\gamma\tau \cos(\psi -x)\right)} \, .
\end{eqnarray*} 
\end{itemize}
\end{itemize}

If we use the solution (\ref{tau:old:ipci}) to Equation $(\ref{AB})_{\alpha\beta}$  we obtain
$$
A(x)=-\frac{\beta\cos(\psi-x)(\rho(-1+\sigma)-\sigma)}{\gamma\rho\sigma} \, .
$$
According to \cite{BCLMsecond}, the assumption SM5 implies that the PDE 
\begin{equation} \label{pde-SM5}
-\left( \frac{\partial V}{\partial s}+\frac{\partial V_\epsilon}{\partial s} \right)\left( -\frac{1}{\sigma} + \frac{\rho-1}{\rho} \right)g^{ad}g_{\alpha d}+\frac{\partial V_\epsilon}{\partial x}=0
\end{equation}
admits a solution $V_\epsilon$, which provides the proper adjustment of the potential energy term for the controlled Lagrangian $L_{\tau,\sigma,\rho,\epsilon}$. In this example, if we take a derivative of (\ref{pde-SM5}) with respect to $s$ we obtain (\ref{PDE-potencial}). Therefore as a particular solution to (\ref{PDE-potencial}) we recover $V_{\epsilon}=\frac{\epsilon d \gamma^2 y^2}{2\beta^2}$, which is the solution to (\ref{pde-SM5}) given in \cite{BCLMsecond}, where $y=s+\left(-\frac{1}{\sigma}+\frac{\rho-1}{\rho}\right)\frac{\beta}{\gamma}(\sin(x-\psi)+\sin(\psi))$.

On the other hand if we use the solution (\ref{tau:new:ipci}) to Equation $(\ref{AB})_{\alpha\beta}$ then we get
\begin{eqnarray*}
A(x) &=& \frac{\beta(\gamma^2 k^2 (\rho +\sigma) - \rho +1) \cos(\psi -x)\left(\alpha\gamma-\beta^2 \cos^2(\psi -x)\right)}{\gamma\rho \left(-\beta\gamma  k \cos(\psi-x) \sqrt{\alpha\gamma -\beta^2 \cos^2(\psi -x)}+\alpha\gamma-\beta^2 \cos^2(\psi -x)\right)} \\
&& +\frac{\gamma k \rho \sqrt{\alpha\gamma -\beta^2 \cos^2(\psi -x)} \left(2 \beta^2 \cos ^2(\psi -x)-\alpha  \gamma \right)
}{\gamma  \rho  \left(-\beta  \gamma  k \cos(\psi-x) \sqrt{\alpha  \gamma -\beta^2 \cos^2(\psi -x)}+\alpha  \gamma -\beta^2 \cos ^2(\psi -x)\right)} \, .
\end{eqnarray*} 

In this case we can consider a solution to (\ref{PDE-potencial}) of the form
\begin{equation} \label{new-Veps}
V_\epsilon=\gamma g\sin(\psi)s+\frac{1}{2}s^2-sh(x)+G(x)-s_0s + s_0h(x) \, ,
\end{equation}
where $h(x)=\int_0^x A(r)dr$, $s_0$ is a constant and we assume $G'(0)=0$.
Then the potential energy term for $L_{\tau,\sigma,\rho,\epsilon}$ becomes
$$
V_T:=V+V_\epsilon=-d\cos(x)+\frac{1}{2}s^2-sh(x)+G(x)-s_0s + s_0h(x) 
$$
and we have
\begin{equation*}
	\left.\frac{\partial V_T}{\partial x} \right|_{x=0,s=s_0}=G'(0)=0 \, , \quad
	\left.\frac{\partial V_T}{\partial s} \right|_{x=0,s=s_0}=h(0)=0 \, ,
\end{equation*}
that is, with appropriate choice of $V_\epsilon$ we can get that $(x=0,s=s_0)$ is a critical point of $V_T$. We also have
$$
\left.
\left(
\begin{array}{cc}
	\frac{\partial^2 V_T}{\partial x^2} & \frac{\partial^2 V_T}{\partial x\partial s} \\
	\frac{\partial^2 V_T}{\partial x\partial s} & \frac{\partial^2 V_T}{\partial s^2}
\end{array}
\right)
\right|_{(x=0,s=s_0)}
=
\left(
\begin{array}{cc}
d+G''(0) & -A(0) \\
-A(0) & 1
\end{array}
\right) \, .
$$
If we take $G(x)=cx^2$ for some constant $c$, then it is enough to choose $c>\frac{-d+A(0)^2}{2}$ in order to ensure that the above matrix is positive-definite.

On the other hand, notice that the multipliers corresponding to the controlled Lagrangian $L_{\tau,\sigma,\rho,\epsilon}$ are
\begin{eqnarray*}
\bar{g}_{11}&=& \alpha+\frac{\beta^2(\rho-1)\cos^2(\psi-x)}{\gamma}+2\beta\rho\cos(\psi-x)\tau+\gamma(\rho+\sigma)\tau^2 \, , \\
\bar{g}_{12}&=& \rho(\beta\cos(\psi-x)+\gamma\tau) \, , \\
\bar{g}_{22}&=& \rho\gamma \, . 
\end{eqnarray*}

If $D=g_{11}g_{22}-g_{12}^2$ then $\bar{D}=\bar{g}_{11}\bar{g}_{22}-\bar{g}_{12}^2=\rho(D+\sigma(g_{22}\tau)^2)$. As in the previous example, it is enough to choose $k,\sigma>0$ and $\rho>1$ in order to ensure that $(\bar{g}_{ij})$ is positive-definite.

Summing up, the new energy function, which is 
\begin{eqnarray*}
E_{L_{\tau,\sigma,\rho,\epsilon}} &=& \frac{1}{2}(\alpha\dot{x}^2+2\beta\cos(x-\psi)\dot{x}(\dot{s}+\tau\dot{x})+\gamma(\dot{s}+\tau\dot{x})^2) - d\cos(x) -\gamma g \sin(\psi)s \\
&& +\frac{1}{2}\sigma\gamma\tau^2\dot{x}^2 +\frac{1}{2}(\rho-1)\gamma\left( \dot{s} +\frac{\beta}{\gamma}\cos(x-\psi)\dot{x} + \tau\dot{x} \right)^2 + V_{\epsilon}(x,s) \, ,
\end{eqnarray*}
provides a Lyapunov function for the controlled system (\ref{SODEipci-x})-(\ref{SODEipci-s}) if we make the appropriate choices of $V_\epsilon$ and parameters $k$, $\sigma$ and $\rho$. 
More precisely we can choose for instance $V_\epsilon$ given in (\ref{new-Veps}), $k,\sigma>0$ and $\rho>1$. Then $E_{L_{\tau,\sigma,\rho,\epsilon}} $ is a Lyapunov function and therefore (\ref{control-ipci}) is a stabilizing control.

\section{Conclusions and future directions}

In this paper we have used the Helmholtz conditions, which are necessary and sufficient conditions for a SODE to be equivalent to a system of Euler-Lagrange equations, in order to recover the matching conditions given in \cite{BLMfirst}. Using the Helmholtz conditions we can also obtain new matching conditions for a class of mechanical systems. This strategy may also be used in order to derive matching conditions in other situations. For instance, in Example \ref{ex:incline} we have shown by ad hoc computations how to obtain a new stabilizing control for the inverted pendulum on a cart on an incline.  This suggests that analogous computations to the ones given in Appendix \ref{appendix} may also be  carried out for $L_{\tau,\sigma,\rho,\epsilon}$ in order to obtain a new $\tau$ solution and a more general PDE for $V_{\epsilon}$. The matching techniques have also been studied for discrete systems  \cite{BLMZdiscrete} and Euler-Poincar\'{e} systems \cite{BLMpoincare}, but there are also Helmholtz conditions available in these settings \cite{CM2008,BFM2,BFFM}, pointing to a possible application of these Helmholtz conditions to the problem of stabilization of an equilibrium.

\appendix

\section{Matching conditions M1-M3 from Helmholtz conditions} \label{appendix}

Here we provide the detailed computations proving the statement at the end of Section  \ref{sec:ArbDimSpecial}, regarding the vanishing of the implicit Helmholtz conditions (\ref{BB})-(\ref{AA}).

We will need
\begin{eqnarray}
\frac{\partial \tilde{\Phi}_c}{\partial \dot{x}^\beta} &=& \left( g_{\beta c,\gamma}+g_{\gamma c,\beta} + (g_{cd}\tau^d_\beta)_{,\gamma} + (g_{cd}\tau^d_\gamma)_{,\beta} \right)\dot{x}^\gamma +g_{cd,\beta} \dot{\theta}^d \, , \label{phi-c-dbeta} \\
\frac{\partial \tilde{\Phi}_a}{\partial x^\alpha} &=& (g_{\nu a, \gamma\alpha}+(g_{ad}\tau^d_\nu)_{\alpha\gamma})\dot{x}^\gamma \dot{x}^\nu+g_{ab,\gamma\alpha}\dot{x}^\gamma\dot{\theta}^b +g_{ab,\alpha}\ddot{\theta}^b  \nonumber \\
&& +(g_{\gamma a ,\alpha}+g_{ad,\alpha}\tau^d_\gamma+g_{ad}\tau^d_{\gamma,\alpha})\ddot{x}^\gamma \, ,  \label{phi-a-alpha} \\
\frac{\partial \tilde{\Phi}_\nu}{\partial \dot{x}^\beta}&=&(g_{\nu\beta,\gamma} + g_{\nu\gamma,\beta}-g_{\gamma\beta,\nu})\dot{x}^\gamma+(g_{\nu a,\beta}-g_{\beta a,\nu})\dot{\theta}^a \, , \label{phi-nu-beta} \\
\frac{\partial \tilde{\Phi}_\gamma }{\partial x^\alpha}&=& g_{\gamma \nu,\eta \alpha}\dot{x}^\eta \dot{x}^\nu+g_{\gamma a,\nu\alpha}\dot{x}^\nu\dot{\theta^a}+g_{\gamma\nu,\alpha}\ddot{x}^\nu + g_{\gamma a,\alpha}\ddot{\theta}^a  \nonumber \\
&& -\frac{1}{2}g_{\nu\eta,\gamma\alpha}\dot{x}^\nu \dot{x}^\eta -g_{\nu a,\gamma \alpha}\dot{x}^\nu \dot{\theta}^a - \frac{1}{2}g_{ab,\gamma\alpha}\dot{\theta}^a\dot{\theta}^b + \frac{\partial^2 V}{\partial x^\alpha \partial x^\gamma} \, ,  \label{phi-gamma-alpha}
\end{eqnarray} 
and also the following expressions:
\begin{eqnarray}
\frac{\partial \tilde{F}_\alpha}{\partial \dot{x}^\gamma}\tilde{W}^{\gamma \nu} +\frac{\partial \tilde{F}_\alpha}{\partial \dot{\theta}^a}\tilde{W}^{a \nu} &=& (g_{\alpha\gamma}+g_{\alpha d}\tau^d_\gamma + g_{\gamma d}\tau^d_\alpha +g_{fd}\tau^f_\alpha \tau^d_\gamma +\sigma_{fd}\tau^f_\alpha \tau^d_\gamma )A^{\gamma\nu} \nonumber \\
&& -(g_{\alpha a}+g_{da}\tau^d_\alpha)\left( g^{ae}g_{e\gamma}+\tau^a_\gamma \right)A^{\gamma\nu} \nonumber \\
&=& (g_{\alpha\gamma}-g_{\alpha a}g^{ae}g_{e\gamma} +\sigma_{ad}\tau^a_\alpha\tau^d_\gamma )A^{\gamma\nu}  \nonumber\\
&=& \delta_{\alpha}^{\nu}+(g_{\alpha d}\tau^d_\gamma+\sigma_{ad}\tau^a_\alpha\tau^d_\gamma)A^{\gamma\nu} \, ,  \label{FW1}\\
\frac{\partial \tilde{F}_\alpha}{\partial \dot{x}^\gamma}\tilde{W}^{\gamma c} +\frac{\partial \tilde{F}_\alpha}{\partial \dot{\theta}^a}\tilde{W}^{a c} &=& (g_{\alpha\gamma}+g_{\alpha d}\tau^d_\gamma + g_{\gamma d}\tau^d_\alpha +g_{fd}\tau^f_\alpha \tau^d_\gamma +\sigma_{fd}\tau^f_\alpha \tau^d_\gamma )\left( -A^{\gamma\eta}g_{\eta e}g^{ec} \right) \nonumber \\
&& +(g_{\alpha a}+g_{da}\tau^d_\alpha)\left( g^{ac}+(g^{af}g_{f\eta}+\tau^a_\eta)A^{\eta\nu}g_{\nu e}g^{ec} \right) \nonumber \\
&=& (-g_{\alpha\gamma}-\sigma_{ad}\tau^a_\alpha\tau^d_\gamma+g_{\alpha a}g^{af}g_{f\gamma})A^{\gamma \eta}g_{\eta e}g^{ec}+g_{\alpha a}g^{ac}+\tau^c_\alpha \nonumber \\
&=& \tau^c_\alpha - (g_{\alpha d}\tau^d_\gamma+\sigma_{ad}\tau^a_\alpha \tau^d_\gamma)A^{\gamma\eta}g_{\eta e}g^{ec} \, . \label{FW2}
\end{eqnarray}

\underline{Equation (\ref{BB}) vanishes  identically for all indices}: 
\begin{eqnarray*}
(\ref{BB})_{a b} & = & \frac{\partial \tilde{F}_a}{\partial \dot{\theta}^b}-\frac{\partial \tilde{F}_b}{\partial \dot{\theta}^a} = g_{ab}-g_{ba}=0 \, ,\\
(\ref{BB})_{a \beta} & = & \frac{\partial \tilde{F}_a}{\partial \dot{x}^\beta}-\frac{\partial \tilde{F}_\beta}{\partial \dot{\theta}^a} = g_{\beta a}+g_{ad}\tau^d_\beta - (g_{\beta a}+g_{da}\tau^d_\beta) =0 \, , \\
(\ref{BB})_{\alpha \beta} & = & \frac{\partial \tilde{F}_\alpha}{\partial \dot{x}^\beta}-\frac{\partial \tilde{F}_\beta}{\partial \dot{x}^\alpha}=(g_{\alpha \beta}+ g_{\alpha a}\tau^a_\beta +g_{\beta a}\tau^a_\alpha+g_{ab}\tau^a_\alpha\tau^b_\beta+\sigma_{ab}\tau^a_\alpha\tau^b_\beta) \\
&& -(g_{\beta \alpha}+ g_{\beta a}\tau^a_\alpha +g_{\alpha a}\tau^a_\beta+g_{ab}\tau^a_\beta\tau^b_\alpha+\sigma_{ab}\tau^a_\beta\tau^b_\alpha) =0 \, .
\end{eqnarray*}

\underline{Equation (\ref{AB}) vanishes identically for indices $ab$ and $a\beta$}:
\begin{eqnarray*}
(\ref{AB})_{ab} & = & \frac{\partial^2 \tilde{F}_a}{\partial \dot{\theta}^b \partial q^k}\dot{q}^k+\frac{\partial \tilde{F}_a}{\partial \theta^b}+\frac{\partial^2 \tilde{F}_a}{\partial \dot{\theta}^b \partial \dot{q}^k}\ddot{q}^k-\frac{\partial \tilde{F}_b}{\partial \theta^a} - \frac{\partial \tilde{F}_a}{\partial \dot{q}^k}\frac{\partial \tilde{\Phi}_r}{\partial \dot{\theta}^b}\tilde{W}^{k r} \\
&=& \frac{\partial^2 \tilde{F}_a}{\partial \dot{\theta}^b \partial x^\gamma}\dot{x}^\gamma - \frac{\partial \tilde{F}_a}{\partial \dot{x}^\gamma} \left(\frac{\partial \Phi_c}{\partial \dot{\theta}^b}\tilde{W}^{\gamma c} + \frac{\partial \Phi_\beta}{\partial \dot{\theta}^b}\tilde{W}^{\gamma \beta} \right) - \frac{\partial \tilde{F}_a}{\partial \dot{\theta}^d} \left(\frac{\partial \Phi_c}{\partial \dot{\theta}^b}\tilde{W}^{d c} + \frac{\partial \Phi_\beta}{\partial \dot{\theta}^b}\tilde{W}^{d \beta} \right) \\
&=& g_{ab,\gamma}\dot{x}^\gamma -(\tilde{C}_{a \gamma}\tilde{W}^{\gamma c}+\tilde{C}_{a d}\tilde{W}^{d c})\frac{\partial \Phi_c}{\partial \dot{\theta}^b} -(\tilde{C}_{a \gamma}\tilde{W}^{\gamma \beta}+\tilde{C}_{a d}\tilde{W}^{d \beta})\frac{\partial \Phi_\beta}{\partial \dot{\theta}^b} \\
&=& g_{ab,\gamma}\dot{x}^\gamma - \frac{\partial \Phi_a}{\partial \dot{\theta}^b} =0 \\
(\ref{AB})_{a\beta} & = & \frac{\partial^2 \tilde{F}_a}{\partial \dot{x}^\beta \partial q^k}\dot{q}^k+\frac{\partial \tilde{F}_a}{\partial x^\beta}+\frac{\partial^2 \tilde{F}_a}{\partial \dot{x}^\beta \partial \dot{q}^k}\ddot{q}^k-\frac{\partial \tilde{F}_\beta}{\partial \theta^a} - \frac{\partial \tilde{F}_a}{\partial \dot{q}^k}\frac{\partial \tilde{\Phi}_r}{\partial \dot{x}^\beta}\tilde{W}^{k r} \\
&=&(g_{\beta a}+g_{ab}\tau^b_\beta)_{,\gamma}\dot{x}^\gamma + (g_{\gamma a}+g_{ab}\tau^b_\gamma)_{,\beta}\dot{x}^\gamma +g_{ab,\beta}\dot{\theta}^b \\
&& - \frac{\partial \tilde{F}_a}{\partial \dot{x}^\gamma} \left(\frac{\partial \tilde{\Phi}_c}{\partial \dot{x}^\beta}\tilde{W}^{\gamma c} + \frac{\partial \Phi_\nu}{\partial \dot{x}^\beta}\tilde{W}^{\gamma \nu} \right) - \frac{\partial \tilde{F}_a}{\partial \dot{\theta}^d} \left(\frac{\partial \tilde{\Phi}_c}{\partial \dot{x}^\beta}\tilde{W}^{d c} + \frac{\partial \Phi_\nu}{\partial \dot{x}^\beta}\tilde{W}^{d \nu} \right) \\
&=&(g_{\beta a}+g_{ab}\tau^b_\beta)_{,\gamma}\dot{x}^\gamma + (g_{\gamma a}+g_{ab}\tau^b_\gamma)_{,\beta}\dot{x}^\gamma +g_{ab,\beta}\dot{\theta}^b \\
&& -(\tilde{C}_{a \gamma}\tilde{W}^{\gamma c}+\tilde{C}_{a d}\tilde{W}^{d c})\frac{\partial \tilde{\Phi}_c}{\partial \dot{x}^\beta} -(\tilde{C}_{a \gamma}\tilde{W}^{\gamma \nu}+\tilde{C}_{a d}\tilde{W}^{d \nu})\frac{\partial \Phi_\nu}{\partial \dot{x}^\beta} \\
&=&(g_{\beta a}+g_{ab}\tau^b_\beta)_{,\gamma}\dot{x}^\gamma + (g_{\gamma a}+g_{ab}\tau^b_\gamma)_{,\beta}\dot{x}^\gamma +g_{ab,\beta}\dot{\theta}^b -\frac{\partial \tilde{\Phi}_a}{\partial \dot{x}^\beta} \\
&\stackrel{(\ref{phi-c-dbeta})}{=}& 0 \, .
\end{eqnarray*}
\uline{For indices $\alpha b$ Equation (\ref{AB}) vanishes using M1 and M3 . Alternatively it vanishes if $(g_{ab})$ is constant and the system has one degree of underactuation}:
\begin{eqnarray*}
(\ref{AB})_{\alpha b} & = & \frac{\partial^2 \tilde{F}_\alpha}{\partial \dot{\theta}^b \partial q^k}\dot{q}^k+\frac{\partial \tilde{F}_\alpha}{\partial \theta^b}+\frac{\partial^2 \tilde{F}_\alpha}{\partial \dot{\theta}^b \partial \dot{q}^k}\ddot{q}^k-\frac{\partial \tilde{F}_b}{\partial x^\alpha} - \frac{\partial \tilde{F}_\alpha}{\partial \dot{q}^k}\frac{\partial \tilde{\Phi}_r}{\partial \dot{\theta}^b}\tilde{W}^{k r} \\
& = & \frac{\partial^2 \tilde{F}_\alpha}{\partial \dot{\theta}^b \partial x^\gamma}\dot{x}^\gamma-\frac{\partial \tilde{F}_b}{\partial x^\alpha} - \frac{\partial \tilde{F}_\alpha}{\partial \dot{x}^\gamma}\left(\frac{\partial \tilde{\Phi}_c}{\partial \dot{\theta}^b}\tilde{W}^{\gamma c} +\frac{\partial \tilde{\Phi}_\nu}{\partial \dot{\theta}^b}\tilde{W}^{\gamma \nu} \right)
-  \frac{\partial \tilde{F}_\alpha}{\partial \dot{\theta}^a}\left( \frac{\partial \tilde{\Phi}_c}{\partial \dot{\theta}^b}\tilde{W}^{a c} +\frac{\partial \tilde{\Phi}_\nu}{\partial \dot{\theta}^b}\tilde{W}^{a \nu}  \right) \\
& = & \frac{\partial^2 \tilde{F}_\alpha}{\partial \dot{\theta}^b \partial x^\gamma}\dot{x}^\gamma-\frac{\partial \tilde{F}_b}{\partial x^\alpha}- \left( \frac{\partial \tilde{F}_\alpha}{\partial \dot{x}^\gamma}\tilde{W}^{\gamma c} +\frac{\partial \tilde{F}_\alpha}{\partial \dot{\theta}^a}\tilde{W}^{a c} \right)\frac{\partial \tilde{\Phi}_c}{\partial \dot{\theta}^b} - \left( \frac{\partial \tilde{F}_\alpha}{\partial \dot{x}^\gamma}\tilde{W}^{\gamma \nu} +\frac{\partial \tilde{F}_\alpha}{\partial \dot{\theta}^a}\tilde{W}^{a \nu} \right)\frac{\partial \tilde{\Phi}_\nu}{\partial \dot{\theta}^b} \\
& \stackrel{(\ref{FW1}),(\ref{FW2})}{=} &  (g_{\alpha b}+g_{db}\tau^d_\alpha)_{,\gamma}\dot{x}^\gamma - (g_{\gamma b}+g_{bd}\tau^d_\gamma)_{,\alpha}\dot{x}^\gamma -g_{bd,\alpha}\dot{\theta}^d\\
&& -\left(\tau^c_\alpha - (g_{\alpha d}\tau^d_\gamma+\sigma_{ad}\tau^a_\alpha \tau^d_\gamma)A^{\gamma\eta}g_{\eta e}g^{ec} \right) g_{cb,\gamma}\dot{x}^\gamma  \\
&& - \left( \delta_{\alpha}^{\nu}+(g_{\alpha d}\tau^d_\gamma+\sigma_{ad}\tau^a_\alpha\tau^d_\gamma)A^{\gamma\nu} \right) \left(\left( g_{\nu b,\gamma}-g_{\gamma b,\nu}\right)\dot{x}^\gamma -g_{db,\nu}\dot{\theta}^d \right) \, .
\end{eqnarray*}
The $\dot{\theta}$ component becomes
\begin{equation*}
\left (-g_{bd,\alpha} + \left( \delta_{\alpha}^{\nu}+(g_{\alpha e}\tau^e_\gamma+\sigma_{ae}\tau^a_\alpha\tau^e_\gamma)A^{\gamma\nu} \right)g_{db,\nu} \right) \dot{\theta}^d = (g_{\alpha e}+\sigma_{ae}\tau^a_\alpha)\tau^e_\gamma A^{\gamma\nu}g_{db,\nu} \dot{\theta}^d \, ,
\end{equation*}
from which we can clearly see the solution M1 (but there are more).
The $\dot{x}$ component becomes
\begin{eqnarray} \label{comp}
&&(g_{\alpha b}+g_{db}\tau^d_\alpha)_{,\gamma}\dot{x}^\gamma - (g_{\gamma b}+g_{bd}\tau^d_\gamma)_{,\alpha}\dot{x}^\gamma
 -\left(\tau^c_\alpha - (g_{\alpha d}\tau^d_\gamma+\sigma_{ad}\tau^a_\alpha \tau^d_\gamma)A^{\gamma\eta}g_{\eta e}g^{ec} \right)  g_{cb,\gamma}\dot{x}^\gamma  \nonumber
\\
&&
- \left( \delta_{\alpha}^{\nu}+(g_{\alpha d}\tau^d_\gamma+\sigma_{ad}\tau^a_\alpha\tau^d_\gamma)A^{\gamma\nu} \right) \left( g_{\nu b,\gamma}-g_{\gamma b,\nu}\right)\dot{x}^\gamma \nonumber \\
&=& \left( g_{db}\left( \tau^d_{\alpha,\gamma} - \tau^d_{\gamma,\alpha} \right) - g_{bd,\alpha}\tau^d_\gamma
+ (g_{\alpha d}+\sigma_{ad}\tau^a_\alpha)\tau^d_\gamma A^{\gamma\eta} \left( g_{\eta e}g^{ec}g_{cb,\gamma}-g_{\eta b,\gamma} +g_{\gamma b,\eta} \right) \right)\dot{x}^\gamma \, ,
\end{eqnarray}
from where assuming M1 we obtain M3 as a solution.

On the other hand notice that if $(g_{ab})$ is constant, that is, SM2 holds, and the system has one degree of underacuation then the equation vanishes identically without imposing M1 nor M3. It is also enough to assume the simplified matching conditions SM2, SM4 and the matching condition M3, without assuming M1.

\underline{For indices $\alpha\beta$ Equation (\ref{AB}) vanishes using M1, M2 and M3}:
\begin{eqnarray*}
(\ref{AB})_{\alpha \beta} & = & \frac{\partial^2 \tilde{F}_\alpha}{\partial \dot{x}^\beta \partial q^k}\dot{q}^k+\frac{\partial \tilde{F}_\alpha}{\partial x^\beta}+\frac{\partial^2 \tilde{F}_\alpha}{\partial \dot{x}^\beta \partial \dot{q}^k}\ddot{q}^k-\frac{\partial \tilde{F}_\beta}{\partial x^\alpha} - \frac{\partial \tilde{F}_\alpha}{\partial \dot{q}^k}\frac{\partial \tilde{\Phi}_r}{\partial \dot{x}^\beta}\tilde{W}^{k r}  \\
&=& \left(-\frac{\partial \tilde{F}_\alpha}{\partial \dot{x}^\gamma}\tilde{W}^{\gamma c} - \frac{\partial \tilde{F}_\alpha }{\partial \dot{\theta}^a } \tilde{W}^{a c} \right)\frac{\partial \tilde{\Phi}_c}{\partial \dot{x}^\beta}+
\left(-\frac{\partial \tilde{F}_\alpha}{\partial \dot{x}^\gamma}\tilde{W}^{\gamma \nu} - \frac{\partial \tilde{F}_\alpha }{\partial \dot{\theta}^a } \tilde{W}^{a \nu} \right)\frac{\partial \tilde{\Phi}_\nu}{\partial \dot{x}^\beta} \\
&& + (g_{\alpha \beta} + g_{\alpha d}\tau^d_\beta + g_{\beta d}\tau^d_\alpha +g_{de}\tau^d_\beta \tau^e_\alpha +\sigma_{ed}\tau^e_\alpha \tau^d_\beta  )_{,\gamma}\dot{x}^\gamma + \frac{\partial \tilde{F}_\alpha}{\partial x^\beta} -\frac{\partial \tilde{F}_\beta}{\partial x^\alpha} \\
& \stackrel{(\ref{FW1}),(\ref{FW2})}{=} & -\left( \tau^c_\alpha - (g_{\alpha d}\tau^d_\gamma+\sigma_{ad}\tau^a_\alpha \tau^d_\gamma)A^{\gamma\eta}g_{\eta e}g^{ec} \right)\frac{\partial \tilde{\Phi}_c}{\partial \dot{x}^\beta}
-\left( \delta_{\alpha}^{\nu}+(g_{\alpha d}\tau^d_\gamma+\sigma_{ad}\tau^a_\alpha\tau^d_\gamma)A^{\gamma\nu} \right)\frac{\partial \tilde{\Phi}_\nu}{\partial \dot{x}^\beta} \\
&& + (g_{\alpha \beta} + g_{\alpha d}\tau^d_\beta + g_{\beta d}\tau^d_\alpha +g_{de}\tau^d_\beta \tau^e_\alpha +\sigma_{ed}\tau^e_\alpha \tau^d_\beta  )_{,\gamma}\dot{x}^\gamma + \frac{\partial \tilde{F}_\alpha}{\partial x^\beta} -\frac{\partial \tilde{F}_\beta}{\partial x^\alpha} \\
&=& -\left( \tau^c_\alpha - (g_{\alpha d}\tau^d_\gamma+\sigma_{ad}\tau^a_\alpha \tau^d_\gamma)A^{\gamma\eta}g_{\eta e}g^{ec} \right) \left( g_{\beta c,\gamma}+g_{\gamma c,\beta} + (g_{cd}\tau^d_\beta)_{,\gamma} + (g_{cd}\tau^d_\gamma)_{,\beta} \right)\dot{x}^\gamma  \\
&&-\left( \tau^c_\alpha - (g_{\alpha d}\tau^d_\gamma+\sigma_{ad}\tau^a_\alpha \tau^d_\gamma)A^{\gamma\eta}g_{\eta e}g^{ec} \right)g_{cf,\beta} \dot{\theta}^f  \\
&& -\left( \delta_{\alpha}^{\nu}+(g_{\alpha d}\tau^d_\gamma+\sigma_{ad}\tau^a_\alpha\tau^d_\gamma)A^{\gamma\nu} \right)  (g_{\nu\beta,\gamma} + g_{\nu\gamma,\beta}-g_{\gamma\beta,\nu})\dot{x}^\gamma  \\
&& -\left( \delta_{\alpha}^{\nu}+(g_{\alpha d}\tau^d_\gamma+\sigma_{ad}\tau^a_\alpha\tau^d_\gamma)A^{\gamma\nu} \right) (g_{\nu f,\beta}-g_{\beta f,\nu})\dot{\theta}^f  \\
&& + (g_{\alpha \beta} + g_{\alpha d}\tau^d_\beta + g_{\beta d}\tau^d_\alpha +g_{de}\tau^d_\beta \tau^e_\alpha +\sigma_{ed}\tau^e_\alpha \tau^d_\beta  )_{,\gamma}\dot{x}^\gamma \\
&& +(g_{\alpha\gamma}+g_{\alpha d}\tau^d_\gamma + g_{\gamma d}\tau^d_\alpha+g_{ab}\tau^a_\alpha\tau^b_\gamma + \sigma_{ab}\tau^a_\alpha \tau^b_\gamma)_{,\beta}\dot{x}^\gamma + (g_{\alpha f} + g_{ef}\tau^e_\alpha )_{,\beta}\dot{\theta}^f \\
&& - (g_{\beta \gamma}+g_{\beta d}\tau^d_\gamma + g_{\gamma d}\tau^d_\beta+g_{ab}\tau^a_\beta \tau^b_\gamma + \sigma_{ab}\tau^a_\beta \tau^b_\gamma)_{,\alpha}\dot{x}^\gamma - (g_{\beta f} + g_{ef}\tau^e_\beta )_{,\alpha}\dot{\theta}^f \, .
\end{eqnarray*}
The $\dot{\theta}$ component becomes
\begin{eqnarray} \label{theta-comp}
\left( (g_{\alpha d}\tau^d_\gamma +\sigma_{ad}\tau^a_\alpha \tau^d_\gamma)A^{\gamma\eta}\left( g_{\eta e}g^{ec}g_{cf,\beta} - g_{\eta f,\beta} + g_{\beta f,\eta} \right) +g_{ef}\left( \tau^e_{\alpha,\beta} - \tau^e_{\beta, \alpha} \right)-g_{ef,\alpha}\tau^e_\beta \right)\dot{\theta}^f \, ,
\end{eqnarray}
which vanishes with the same assumptions as (\ref{comp}).

The $\dot{x}$ component becomes
\begin{eqnarray*}
&& -\left( \tau^c_\alpha - (g_{\alpha d}\tau^d_\gamma+\sigma_{ad}\tau^a_\alpha \tau^d_\gamma)A^{\gamma\eta}g_{\eta e}g^{ec} \right) \left( g_{\beta c,\gamma}+g_{\gamma c,\beta} + (g_{cd}\tau^d_\beta)_{,\gamma} + (g_{cd}\tau^d_\gamma)_{,\beta} \right)\dot{x}^\gamma  \\
&& -\left( \delta_{\alpha}^{\nu}+(g_{\alpha d}\tau^d_\gamma+\sigma_{ad}\tau^a_\alpha\tau^d_\gamma)A^{\gamma\nu} \right)  (g_{\nu\beta,\gamma} + g_{\nu\gamma,\beta}-g_{\gamma\beta,\nu})\dot{x}^\gamma  \\
&& + (g_{\alpha \beta} + g_{\alpha d}\tau^d_\beta + g_{\beta d}\tau^d_\alpha +g_{de}\tau^d_\beta \tau^e_\alpha +\sigma_{ed}\tau^e_\alpha \tau^d_\beta  )_{,\gamma}\dot{x}^\gamma \\
&& +(g_{\alpha\gamma}+g_{\alpha d}\tau^d_\gamma + g_{\gamma d}\tau^d_\alpha+g_{ab}\tau^a_\alpha\tau^b_\gamma + \sigma_{ab}\tau^a_\alpha \tau^b_\gamma)_{,\beta}\dot{x}^\gamma  \\
&& - (g_{\beta \gamma}+g_{\beta d}\tau^d_\gamma + g_{\gamma d}\tau^d_\beta+g_{ab}\tau^a_\beta \tau^b_\gamma + \sigma_{ab}\tau^a_\beta \tau^b_\gamma)_{,\alpha}\dot{x}^\gamma \\
&\stackrel{\mbox{M1}}{=}& (g_{\alpha a}\tau^a_\beta)_{,\gamma}+g_{\beta a}\tau^a_{\alpha,\gamma}+g_{ad}\tau^d_\beta \tau^a_{\alpha,\gamma}+(\sigma_{ad}\tau^a_\alpha \tau^d_\beta)_{,\gamma}+(g_{\alpha d}\tau^d_\gamma)_{,\beta} + g_{\gamma d}\tau^d_{\alpha,\beta}+g_{ab}\tau^b_\gamma \tau^a_{\alpha,\beta} \\
&&+(\sigma_{ad}\tau^a_\alpha \tau^d_\gamma)_{,\beta}-(g_{\beta d}\tau^d_\gamma)_{,\alpha}-(g_{\gamma d}\tau^d_\beta)_{,\alpha}-(g_{ab}\tau^a_\beta \tau^b_\gamma)_{,\alpha}-(\sigma_{ad}\tau^a_\beta \tau^d_\gamma)_{,\alpha} \\
&=& g_{ab}\tau^b_\gamma (\tau^a_{\alpha,\beta}-\tau^a_{\beta,\alpha}) +g_{\beta a} (\tau^a_{\alpha,\gamma}-\tau^a_{\gamma,\alpha}) +g_{ad}\tau^d_{\beta} (\tau^a_{\alpha,\gamma}-\tau^a_{\gamma,\alpha}) +g_{\gamma d} \tau^d_{\alpha,\beta} -g_{\beta d,\alpha}\tau^d_\gamma-g_{ab,\alpha}\tau^a_\beta \tau^b_\gamma
\end{eqnarray*}
where we have used again M1 and the symmetry of $\sigma_{ad}$. Now using M3 to cancel the first and last terms  we have
\begin{eqnarray*}
&& (g_{\beta a}+g_{ad}\tau^d_\beta)(\tau^a_{\alpha,\gamma}-\tau^a_{\gamma,\alpha})+g_{\gamma d}\tau^d_{\alpha,\beta}-g_{\beta d,\alpha}\tau^d_\gamma \\
&\stackrel{\mbox{M1,M3}}{=}& (g_{\beta a}+g_{ad}\tau^d_\beta)(g^{ea}g_{fe,\alpha}\tau^f_\gamma)+g_{\gamma d}\tau^d_{\alpha,\beta}+(\sigma_{de}\tau^e_\beta)_{,\alpha}\tau^d_\gamma \\
&\stackrel{\mbox{M1}}{=}& g_{\beta a}g^{ea}g_{fe,\alpha}\tau^f_\gamma + \tau^d_\beta g_{fd,\alpha}\tau^f_\gamma-\sigma_{de}\tau^e_\gamma\tau^d_{\alpha,\beta}+\sigma_{de,\alpha}\tau^e_\beta \tau^d_\gamma+\sigma_{de}\tau^e_{\beta,\alpha}\tau^d_\gamma \\
&\stackrel{\mbox{M1}}{=}& -\sigma_{ah}\tau^h_\beta g^{ea}g_{fe,\alpha}\tau^f_\gamma + \tau^d_\beta g_{fd,\alpha}\tau^f_\gamma+\sigma_{de,\alpha}\tau^e_\beta \tau^d_\gamma+\sigma_{de}\tau^d_\gamma(\tau^e_{\beta,\alpha}-\tau^e_{\alpha,\beta}) \\
&\stackrel{\mbox{M3}}{=}& -\sigma_{ah}\tau^h_\beta g^{ea}g_{fe,\alpha}\tau^f_\gamma - \sigma_{de}\tau^d_\gamma g^{fe}g_{hf,\alpha}\tau^h_\beta + (g_{de,\alpha}+\sigma_{de,\alpha})\tau^d_\gamma \tau^e_\beta \, .
\end{eqnarray*}
Summing up, for the $\dot{x}$ component, using M1 and M3 we get
\begin{equation} \label{genM2}
(g_{de,\alpha}+\sigma_{de,\alpha})\tau^d_\gamma \tau^e_\beta -\sigma_{ah}\tau^h_\beta g^{ea}g_{fe,\alpha}\tau^f_\gamma - \sigma_{de}\tau^d_\gamma g^{fe} g_{hf,\alpha} \tau^h_\beta \, .
\end{equation}
If we use M2, (\ref{genM2}) becomes
\begin{eqnarray*}
& & (g_{de,\alpha}+\sigma_{de,\alpha})\tau^d_\gamma \tau^e_\beta -\sigma_{ah}\tau^h_\beta \frac{1}{2}\sigma^{ea}(\sigma_{de,\alpha}+g_{de,\alpha})\tau^d_\gamma - \sigma_{de}\tau^d_\gamma \frac{1}{2}\sigma^{fe}(\sigma_{hf,\alpha} + g_{hf,\alpha} ) \tau^h_\beta \\
&=& (g_{de,\alpha}+\sigma_{de,\alpha})\tau^d_\gamma \tau^e_\beta -  \frac{1}{2}(\sigma_{dh,\alpha}+g_{dh,\alpha})\tau^d_\gamma \tau^h_\beta - \frac{1}{2}(\sigma_{hd,\alpha} + g_{hd,\alpha} ) \tau^h_\beta \tau^d_\gamma = 0 \, .
\end{eqnarray*}
That is using all of M1, M2 and M3 we get that the Helmholtz condition $(\ref{AB})_{\alpha \beta}$ vanishes.

\underline{Equation (\ref{AA}) vanishes identically for indices $ab$ and $\alpha b$}:
\begin{eqnarray*}
(\ref{AA})_{a b} & = & 
\frac{\partial^2 \tilde{F}_a}{\partial \theta^b \partial q^k}\dot{q}^k+\frac{\partial^2 \tilde{F}_a}{\partial \theta^b \partial \dot{q}^k}\ddot{q}^k-\frac{\partial \tilde{F}_a}{\partial \dot{q}^k}\frac{\partial \tilde{\Phi}_r}{\partial \theta^b}\tilde{W}^{k r} - \frac{\partial^2 \tilde{F}_b}{\partial \theta^a\partial q^k}\dot{q}^k - \frac{\partial^2 \tilde{F}_b}{\partial \theta^a \partial \dot{q}^k}\ddot{q}^k + \frac{\partial \tilde{F}_b}{\partial \dot{q}^k}\frac{\partial \tilde{\Phi}_r}{\partial \theta^a}\tilde{W}^{k r} =0 \, , \\
(\ref{AA})_{\alpha b} & = & \frac{\partial^2 \tilde{F}_\alpha}{\partial \theta^b \partial q^k}\dot{q}^k+\frac{\partial^2 \tilde{F}_\alpha}{\partial \theta^b \partial \dot{q}^k}\ddot{q}^k-\frac{\partial \tilde{F}_\alpha}{\partial \dot{q}^k}\frac{\partial \tilde{\Phi}_r}{\partial \theta^b}\tilde{W}^{k r} - \frac{\partial^2 \tilde{F}_b}{\partial x^\alpha \partial q^k}\dot{q}^k - \frac{\partial^2 \tilde{F}_b}{\partial x^\alpha \partial \dot{q}^k}\ddot{q}^k + \frac{\partial \tilde{F}_b}{\partial \dot{q}^k}\frac{\partial \tilde{\Phi}_r}{\partial x^\alpha}\tilde{W}^{k r} \\
& = & - \frac{\partial^2 \tilde{F}_b}{\partial x^\alpha \partial x^\gamma}\dot{x}^\gamma - \frac{\partial^2 \tilde{F}_b}{\partial x^\alpha \partial \dot{q}^k}\ddot{q}^k + \frac{\partial \tilde{F}_b}{\partial \dot{q}^k}\left( \frac{\partial \tilde{\Phi}_\nu}{\partial x^\alpha}\tilde{W}^{k \nu} + \frac{\partial \tilde{\Phi}_d}{\partial x^\alpha}\tilde{W}^{k d}  \right) \\
& = & - \frac{\partial^2 \tilde{F}_b}{\partial x^\alpha \partial x^\gamma}\dot{x}^\gamma - \frac{\partial^2 \tilde{F}_b}{\partial x^\alpha \partial \dot{q}^k}\ddot{q}^k + \left( \tilde{C}_{b \gamma} \tilde{W}^{\gamma \nu} + \tilde{C}_{b c}\tilde{W}^{c \nu}  \right)\frac{ \partial \tilde{\Phi}_\nu}{\partial x^\alpha} + \left( \tilde{C}_{b \gamma} \tilde{W}^{\gamma d} + \tilde{C}_{b c}\tilde{W}^{c d} \right)\frac{ \partial \tilde{\Phi}_d}{\partial x^\alpha}  \\
& = & - \frac{\partial^2 \tilde{F}_b}{\partial x^\alpha \partial x^\gamma}\dot{x}^\gamma - \frac{\partial^2 \tilde{F}_b}{\partial x^\alpha \partial \dot{x}^\gamma}\ddot{x}^\gamma - \frac{\partial^2 \tilde{F}_b}{\partial x^\alpha \partial \dot{\theta}^c}\ddot{\theta}^c  + \frac{\partial \tilde{\Phi}_b}{\partial x^\alpha}  \\
& = & -(g_{\nu b}+g_{bc}\tau^c_\nu)_{,\alpha\gamma} \dot{x}^\nu \dot{x}^\gamma -g_{bc,\alpha\gamma}\dot{x}^\gamma \dot{\theta}^c - (g_{\gamma b}+g_{bc}\tau^c_\gamma)\ddot{x}^\gamma -g_{bc,\alpha}\ddot{\theta}^c  + \frac{\partial \tilde{\Phi}_b}{\partial x^\alpha}  \\
&\stackrel{(\ref{phi-a-alpha})}{=} &  0 \, .
\end{eqnarray*}
\uline{For indices $\alpha\beta$ Equation (\ref{AA}) vanishes using M1, M2 and M3 or alternatively for systems with one degree of underactuation, since it  is symmetric in $\alpha$ and $\beta$}:
\begin{eqnarray*}
(\ref{AA})_{\alpha \beta} & = & \frac{\partial^2 \tilde{F}_\alpha}{\partial x^\beta \partial x^\gamma}\dot{x}^\gamma + \frac{\partial^2 \tilde{F}_\alpha}{\partial x^\beta \partial \dot{q}^k}\dot{v}^k - \frac{\partial^2 \tilde{F}_\beta}{\partial x^\alpha \partial x^\gamma}\dot{x}^\gamma - \frac{\partial^2 \tilde{F}_\beta}{\partial x^\alpha \partial \dot{q}^k}\dot{v}^k \\
&& -\frac{\partial \tilde{F}_\alpha}{\partial \dot{q}^k}\frac{\partial \tilde{\Phi}_r}{\partial x^\beta}\tilde{W}^{k r} + \frac{\partial \tilde{F}_\beta}{\partial \dot{q}^k}\frac{\partial \tilde{\Phi}_r}{\partial x^\alpha}\tilde{W}^{k r} \\
&\stackrel{\mbox{M1}}{=}& (g_{\alpha b}+g_{ab}\tau^a_\alpha)_{,\beta\gamma}\dot{\theta}^b\dot{x}^\gamma + (g_{\alpha\nu}+g_{\alpha a}\tau^a_\nu + g_{\nu a}\tau^a_\alpha + g_{ab}\tau^a_\alpha\tau^b_\nu +\sigma_{ab}\tau^a_\alpha\tau^b_\nu)_{,\beta \gamma}\dot{x}^\nu\dot{x}^\gamma \\
&& + (g_{\alpha b} +g_{ab}\tau^a_\alpha)_{,\beta}\ddot{\theta}^b + (g_{\alpha\gamma}+g_{\alpha a}\tau^a_\gamma +g_{\gamma a}\tau^a_\alpha + g_{ab}\tau^a_\alpha\tau^b_\gamma+\sigma_{ab}\tau^a_\alpha\tau^b_\gamma)_{,\beta}\ddot{x}^\gamma \\
&& - (g_{\beta b}+g_{ab}\tau^a_\beta)_{,\alpha\gamma}\dot{\theta}^b\dot{x}^\gamma - (g_{\beta\nu}+g_{\beta a}\tau^a_\nu + g_{\nu a}\tau^a_\beta + g_{ab}\tau^a_\beta\tau^b_\nu +\sigma_{ab}\tau^a_\beta\tau^b_\nu)_{,\alpha \gamma}\dot{x}^\nu\dot{x}^\gamma \\
&& - (g_{\beta b} +g_{ab}\tau^a_\beta)_{,\alpha}\ddot{\theta}^b - (g_{\beta\gamma}+g_{\beta a}\tau^a_\gamma +g_{\gamma a}\tau^a_\beta + g_{ab}\tau^a_\beta\tau^b_\gamma+\sigma_{ab}\tau^a_\beta\tau^b_\gamma)_{,\alpha}\ddot{x}^\gamma \\
&& -\frac{\partial \Phi_\alpha}{\partial x^\beta}-\tau^c_\alpha \frac{\partial \tilde{\Phi}_c}{\partial x^\beta} +\frac{\partial \Phi_\beta}{\partial x^\alpha} + \tau^c_\beta \frac{\partial \tilde{\Phi}_c}{\partial x^\alpha}\\
&\stackrel{(\ref{phi-a-alpha}),(\ref{phi-gamma-alpha})}{=}& (g_{\alpha b}+g_{ab}\tau^a_\alpha)_{,\beta\gamma}\dot{\theta}^b\dot{x}^\gamma + (g_{\alpha\nu}+g_{\alpha a}\tau^a_\nu + g_{\nu a}\tau^a_\alpha + g_{ab}\tau^a_\alpha\tau^b_\nu +\sigma_{ab}\tau^a_\alpha\tau^b_\nu)_{,\beta \gamma}\dot{x}^\nu\dot{x}^\gamma \\
&& + (g_{\alpha b} +g_{ab}\tau^a_\alpha)_{,\beta}\ddot{\theta}^b + (g_{\alpha\gamma}+g_{\alpha a}\tau^a_\gamma +g_{\gamma a}\tau^a_\alpha + g_{ab}\tau^a_\alpha\tau^b_\gamma+\sigma_{ab}\tau^a_\alpha\tau^b_\gamma)_{,\beta}\ddot{x}^\gamma \\
&& - (g_{\beta b}+g_{ab}\tau^a_\beta)_{,\alpha\gamma}\dot{\theta}^b\dot{x}^\gamma - (g_{\beta\nu}+g_{\beta a}\tau^a_\nu + g_{\nu a}\tau^a_\beta + g_{ab}\tau^a_\beta\tau^b_\nu +\sigma_{ab}\tau^a_\beta\tau^b_\nu)_{,\alpha \gamma}\dot{x}^\nu\dot{x}^\gamma \\
&& - (g_{\beta b} +g_{ab}\tau^a_\beta)_{,\alpha}\ddot{\theta}^b - (g_{\beta\gamma}+g_{\beta a}\tau^a_\gamma +g_{\gamma a}\tau^a_\beta + g_{ab}\tau^a_\beta\tau^b_\gamma+\sigma_{ab}\tau^a_\beta\tau^b_\gamma)_{,\alpha}\ddot{x}^\gamma \\
&& - (g_{\alpha\nu,\eta\beta}\dot{x}^\eta\dot{x}^\nu + g_{\alpha a,\nu\beta}\dot{x}^\nu\dot{\theta}^a + g_{\alpha\nu,\beta}\ddot{x}^\nu + g_{\alpha a,\beta}\ddot{\theta}^a ) \\
&& + (g_{\beta\nu,\eta\alpha}\dot{x}^\eta\dot{x}^\nu + g_{\beta a,\nu\alpha}\dot{x}^\nu\dot{\theta}^a + g_{\beta\nu,\alpha}\ddot{x}^\nu + g_{\beta a,\alpha}\ddot{\theta}^a ) \\
&& -\tau^c_\alpha (g_{\nu c,\gamma\beta} + (g_{cb}\tau^b_\nu)_{,\beta\gamma})\dot{x}^\nu \dot{x}^\gamma -\tau^c_\alpha g_{cb,\gamma\beta}\dot{x}^\gamma\dot{\theta}^b \\
&& -\tau^c_\alpha (g_{\gamma c,\beta} + g_{cb,\beta}\tau^b_\gamma + g_{cb}\tau^b_{\gamma,\beta} )\ddot{x}^\gamma -\tau^c_\alpha g_{cb,\beta}\ddot{\theta}^b \\
&& +\tau^c_\beta (g_{\nu c,\gamma\alpha} + (g_{cb}\tau^b_\nu)_{,\alpha\gamma})\dot{x}^\nu \dot{x}^\gamma +\tau^c_\beta g_{cb,\gamma\alpha}\dot{x}^\gamma\dot{\theta}^b \\
&& + \tau^c_\beta (g_{\gamma c,\alpha} + g_{cb,\alpha}\tau^b_\gamma + g_{cb}\tau^b_{\gamma,\alpha} )\ddot{x}^\gamma +\tau^c_\beta g_{cb,\alpha}\ddot{\theta}^b \, .
\end{eqnarray*}
The $\ddot{\theta}$ component becomes
$$
(g_{\alpha b} +g_{cb}\tau^c_\alpha)_{,\beta}\ddot{\theta}^b  - (g_{\beta b} +g_{cb}\tau^c_\beta)_{,\alpha}\ddot{\theta}^b - g_{\alpha b,\beta}\ddot{\theta}^b + g_{\beta b,\alpha}\ddot{\theta}^b -\tau^c_\alpha g_{cb,\beta}\ddot{\theta}^b +\tau^c_\beta g_{cb,\alpha}\ddot{\theta}^b = g_{cb}(\tau^c_{\alpha,\beta}-\tau^c_{\beta,\alpha})\ddot{\theta}^b \, .
$$
The $\ddot{x}$ component becomes
\begin{eqnarray*}
&& (g_{\alpha\gamma}+g_{\alpha a}\tau^a_\gamma +g_{\gamma a}\tau^a_\alpha + g_{ab}\tau^a_\alpha\tau^b_\gamma+\sigma_{ab}\tau^a_\alpha\tau^b_\gamma)_{,\beta}\ddot{x}^\gamma - (g_{\beta\gamma}+g_{\beta a}\tau^a_\gamma +g_{\gamma a}\tau^a_\beta + g_{ab}\tau^a_\beta\tau^b_\gamma+\sigma_{ab}\tau^a_\beta\tau^b_\gamma)_{,\alpha}\ddot{x}^\gamma  \\
&& - g_{\alpha\gamma,\beta}\ddot{x}^\gamma + g_{\beta\gamma,\alpha}\ddot{x}^\gamma -\tau^c_\alpha (g_{\gamma c,\beta} + g_{cb,\beta}\tau^b_\gamma + g_{cb}\tau^b_{\gamma,\beta} )\ddot{x}^\gamma + \tau^c_\beta (g_{\gamma c,\alpha} + g_{cb,\alpha}\tau^b_\gamma + g_{cb}\tau^b_{\gamma,\alpha} )\ddot{x}^\gamma \\
&=& g_{\alpha a,\beta}\tau^a_\gamma + g_{\alpha a}\tau^a_{\gamma,\beta} + g_{\gamma a}\tau^a_{\alpha,\beta} + g_{ab}\tau^b_\gamma \tau^a_{\alpha,\beta} + (\sigma_{ab}\tau^a_\alpha\tau^b_\gamma)_\beta \\
&& - g_{\beta a,\alpha}\tau^a_\gamma - g_{\beta a}\tau^a_{\gamma,\alpha} - g_{\gamma a}\tau^a_{\beta,\alpha} - g_{ab}\tau^b_\gamma \tau^a_{\beta,\alpha} - (\sigma_{ab}\tau^a_\beta \tau^b_\gamma)_\alpha \\
&=& (g_{\alpha a}\tau^a_\gamma)_{\beta} + g_{\gamma a}\tau^a_{\alpha,\beta} + g_{ab}\tau^b_\gamma \tau^a_{\alpha,\beta} + (\sigma_{ab}\tau^a_\alpha\tau^b_\gamma)_\beta - (g_{\beta a}\tau^a_\gamma)_{\alpha} - g_{\gamma a}\tau^a_{\beta,\alpha} - g_{ab}\tau^b_\gamma \tau^a_{\beta,\alpha} - (\sigma_{ab}\tau^a_\beta \tau^b_\gamma)_\alpha \\
& \stackrel{\mbox{M1}}{=} & (g_{\gamma c}+g_{cb}\tau^b_\gamma)(\tau^c_{\alpha,\beta}-\tau^c_{\beta,\alpha}) \ddot{x}^\gamma \, .
\end{eqnarray*}
The $\dot{\theta}\dot{x}$ component becomes
\begin{eqnarray*}
&& (g_{\alpha b}+g_{ab}\tau^a_\alpha)_{,\beta\gamma}\dot{\theta}^b\dot{x}^\gamma
- (g_{\beta b}+g_{ab}\tau^a_\beta)_{,\alpha\gamma}\dot{\theta}^b\dot{x}^\gamma
-g_{\alpha b,\gamma\beta}\dot{x}^\gamma\dot{\theta}^b \\
&& +g_{\beta b,\gamma\alpha}\dot{x}^\gamma\dot{\theta}^b
-\tau^c_\alpha g_{cb,\gamma\beta}\dot{x}^\gamma\dot{\theta}^b
+\tau^c_\beta g_{cb,\gamma\alpha}\dot{x}^\gamma\dot{\theta}^b \\
&=&  (g_{ab}\tau^a_\alpha)_{,\beta\gamma}\dot{\theta}^b\dot{x}^\gamma
- (g_{ab}\tau^a_\beta)_{,\alpha\gamma}\dot{\theta}^b\dot{x}^\gamma
-\tau^c_\alpha g_{cb,\gamma\beta}\dot{x}^\gamma\dot{\theta}^b
+\tau^c_\beta g_{cb,\gamma\alpha}\dot{x}^\gamma\dot{\theta}^b \\
&=& \left( g_{ca,\gamma}(\tau^c_{\alpha,\beta}-\tau^c_{\beta, \alpha}) + (g_{ca}\tau^c_{\alpha,\gamma})_{,\beta} - (g_{ca}\tau^c_{\beta,\gamma})_{,\alpha} \right) \dot{\theta}^a\dot{x}^\gamma \\
&\stackrel{\mbox{M3}}{=}& g_{ca,\gamma}(\tau^c_{\alpha,\beta}-\tau^c_{\beta, \alpha}) \dot{\theta}^a\dot{x}^\gamma \, ,
\end{eqnarray*}
where in the last equality we have used M3 in the following way:
\begin{eqnarray*}
- (g_{ca}\tau^c_{\beta,\gamma})_{,\alpha} + (g_{ca}\tau^c_{\alpha,\gamma})_{,\beta} &\stackrel{\mbox{M3}}{=}& -(g_{ca}\tau^c_{\gamma,\beta}+g_{ca}g^{cd}g_{ed,\beta}\tau^e_\gamma)_{,\alpha} + (g_{ca}\tau^c_{\gamma,\alpha}+g_{ca}g^{cd}g_{ed,\alpha}\tau^e_\gamma)_{,\beta} \\
&=& -(g_{ca,\alpha}\tau^c_{\gamma,\beta}+g_{ca}\tau^c_{\gamma,\beta\alpha} + g_{ea,\beta\alpha}\tau^e_\gamma + g_{ea,\beta}\tau^e_{\gamma,\alpha})  \\
&& + g_{ca,\beta}\tau^c_{\gamma,\alpha}+g_{ca}\tau^c_{\gamma,\alpha\beta}+g_{ea,\alpha\beta}\tau^e_\gamma+g_{ea,\alpha}\tau^e_{\gamma,\beta} = 0 \, .
\end{eqnarray*}
Finally the $\dot{x}\dot{x}$ component becomes
\begin{eqnarray*}
&& (g_{\alpha\nu}+g_{\alpha a}\tau^a_\nu + g_{\nu a}\tau^a_\alpha + g_{ab}\tau^a_\alpha\tau^b_\nu +\sigma_{ab}\tau^a_\alpha\tau^b_\nu)_{,\beta \gamma}\dot{x}^\nu\dot{x}^\gamma \\
&& - (g_{\beta\nu}+g_{\beta a}\tau^a_\nu + g_{\nu a}\tau^a_\beta + g_{ab}\tau^a_\beta\tau^b_\nu +\sigma_{ab}\tau^a_\beta\tau^b_\nu)_{,\alpha \gamma}\dot{x}^\nu\dot{x}^\gamma \\
&& -g_{\alpha\nu,\gamma\beta}\dot{x}^\gamma\dot{x}^\nu
+g_{\beta\nu,\gamma\alpha}\dot{x}^\gamma\dot{x}^\nu
-\tau^c_\alpha (g_{\nu c,\gamma\beta} + (g_{cb}\tau^b_\nu)_{,\beta\gamma})\dot{x}^\nu \dot{x}^\gamma
+\tau^c_\beta (g_{\nu c,\gamma\alpha} + (g_{cb}\tau^b_\nu)_{,\alpha\gamma})\dot{x}^\nu \dot{x}^\gamma \\
&=& (g_{ab}\tau^b_\nu)_{,\beta}\tau^a_{\alpha,\gamma} + (g_{ab}\tau^b_\nu)_{,\gamma}\tau^a_{\alpha,\beta} +
 g_{ab}\tau^b_\nu \tau^a_{\alpha,\beta\gamma}  + (g_{\alpha a}\tau^a_\nu)_{,\beta\gamma} + g_{\nu a,\beta}\tau^a_{\alpha,\gamma} + g_{\nu a,\gamma}\tau^a_{\alpha,\beta} \\
&& +g_{\nu a}\tau^a_{\alpha,\beta\gamma}+
(\sigma_{ab}\tau^a_\alpha \tau^b_\nu)_{,\beta\gamma}  - (g_{ab}\tau^b_\nu)_{,\alpha}\tau^a_{\beta,\gamma} - (g_{ab}\tau^b_\nu)_{,\gamma}\tau^a_{\beta,\alpha}
 - g_{ab}\tau^b_\nu \tau^a_{\beta,\alpha\gamma} - (g_{\beta a}\tau^a_\nu)_{,\alpha\gamma} \\
&& - g_{\nu a,\alpha}\tau^a_{\beta,\gamma} - g_{\nu a,\gamma}\tau^a_{\beta,\alpha} - g_{\nu a}\tau^a_{\beta,\alpha\gamma} - (\sigma_{ab}\tau^a_\beta \tau^b_\nu)_{,\alpha\gamma} \\
&=& g_{ab}\tau^b_\nu \tau^a_{\alpha,\beta\gamma} - g_{ab}\tau^b_\nu \tau^a_{\beta,\alpha\gamma} + g_{\nu a,\beta}\tau^a_{\alpha,\gamma} + g_{\nu a,\gamma}\tau^a_{\alpha,\beta} + g_{\nu a}\tau^a_{\alpha,\beta\gamma}  - g_{\nu a,\alpha}\tau^a_{\beta,\gamma} - g_{\nu a,\gamma}\tau^a_{\beta,\alpha} - g_{\nu a}\tau^a_{\beta,\alpha\gamma}  \\
&& +(g_{ab}\tau^b_\nu)_{,\beta}\tau^a_{\alpha,\gamma} + (g_{ab}\tau^b_\nu)_{,\gamma}\tau^a_{\alpha,\beta}
- (g_{ab}\tau^b_\nu)_{,\alpha}\tau^a_{\beta,\gamma} - (g_{ab}\tau^b_\nu)_{,\gamma}\tau^a_{\beta,\alpha} \\
&=& ( g_{\nu a,\gamma}+ (g_{ab}\tau^b_\nu)_{,\gamma})(\tau^a_{\alpha,\beta}-\tau^a_{\beta,\alpha}) \\
&& + g_{ab}\tau^b_\nu \tau^a_{\alpha,\beta\gamma} - g_{ab}\tau^b_\nu \tau^a_{\beta,\alpha\gamma} + g_{\nu a,\beta}\tau^a_{\alpha,\gamma}  + g_{\nu a}\tau^a_{\alpha,\beta\gamma}  - g_{\nu a,\alpha}\tau^a_{\beta,\gamma}  - g_{\nu a}\tau^a_{\beta,\alpha\gamma} \\
&&  +(g_{ab}\tau^b_\nu)_{,\beta}\tau^a_{\alpha,\gamma} - (g_{ab}\tau^b_\nu)_{,\alpha}\tau^a_{\beta,\gamma} \, ,
\end{eqnarray*}
and adding all of the components we get
\begin{eqnarray*}
(\ref{AA})_{\alpha \beta}&=& g_{cb}(\tau^c_{\alpha,\beta}-\tau^c_{\beta,\alpha})\ddot{\theta}^b +
(g_{\gamma c}+g_{cb}\tau^b_\gamma)(\tau^c_{\alpha,\beta}-\tau^c_{\beta,\alpha}) \ddot{x}^\gamma +
g_{ca,\gamma}(\tau^c_{\alpha,\beta}-\tau^c_{\beta, \alpha})\dot{\theta}^a\dot{x}^\gamma \\
&& +( g_{\nu a,\gamma}+ (g_{ab}\tau^b_\nu)_{,\gamma})(\tau^a_{\alpha,\beta}-\tau^a_{\beta,\alpha})\dot{x}^\nu\dot{x}^\gamma \\
&& + (g_{ab}\tau^b_\nu \tau^a_{\alpha,\beta\gamma} - g_{ab}\tau^b_\nu \tau^a_{\beta,\alpha\gamma} + g_{\nu a,\beta}\tau^a_{\alpha,\gamma} + g_{\nu a}\tau^a_{\alpha,\beta\gamma}  - g_{\nu a,\alpha}\tau^a_{\beta,\gamma} - g_{\nu a}\tau^a_{\beta,\alpha\gamma} \\
&& +(g_{ab}\tau^b_\nu)_{,\beta}\tau^a_{\alpha,\gamma} - (g_{ab}\tau^b_\nu)_{,\alpha}\tau^a_{\beta,\gamma} )\dot{x}^\nu\dot{x}^\gamma \\
&=& \tilde{\Phi}_c(\tau^c_{\alpha,\beta}-\tau^c_{\beta,\alpha}) + (g_{ab}\tau^b_\nu \tau^a_{\alpha,\beta\gamma} - g_{ab}\tau^b_\nu \tau^a_{\beta,\alpha\gamma} + g_{\nu a,\beta}\tau^a_{\alpha,\gamma} + g_{\nu a}\tau^a_{\alpha,\beta\gamma}  - g_{\nu a,\alpha}\tau^a_{\beta,\gamma} - g_{\nu a}\tau^a_{\beta,\alpha\gamma} \\
&& +(g_{ab}\tau^b_\nu)_{,\beta}\tau^a_{\alpha,\gamma} - (g_{ab}\tau^b_\nu)_{,\alpha}\tau^a_{\beta,\gamma} )\dot{x}^\nu\dot{x}^\gamma =: \tilde{\Phi}_c(\tau^c_{\alpha,\beta}-\tau^c_{\beta,\alpha}) + R\dot{x}^\nu\dot{x}^\gamma   \, .
\end{eqnarray*}
Next we will show that the term $R\dot{x}^\nu\dot{x}^\gamma$ vanishes using  M1, M2 and M3. First we compute some expressions that we will need. We will consistently omit writing the term $\dot{x}^\nu\dot{x}^\gamma$ but will take it into account and cancel any symmetric terms in $\nu$ and $\gamma$. From M2 we get $\sigma_{ab,\beta}=-g_{ab,\beta}+2\sigma_{be}g^{ed}g_{ad,\beta}$
and therefore 
\begin{equation} \label{aux1}
(g_{ab}-\sigma_{ab})_{,\beta}=2g_{ab,\beta}-2\sigma_{be}g^{ed}g_{ad,\beta}\, .
\end{equation} 
Using $g^{da}_{,\alpha}=-g^{ea}g^{dh}g_{eh,\alpha}$ we get
\begin{eqnarray*}
(g^{da}g_{ed,\alpha}\tau^e_{\gamma})_{,\beta} - (g^{da}g_{ed,\beta}\tau^e_{\gamma})_{,\alpha} &=& g^{da}(g_{ed,\alpha\beta}\tau^e_\gamma + g_{ed,\alpha}\tau^e_{\gamma,\beta} - g_{ed,\beta\alpha}\tau^e_\gamma - g_{ed,\beta}\tau^e_{\gamma,\alpha} ) \\
&& +g^{da}_{,\beta}g_{ed,\alpha}\tau^e_\gamma -g^{da}_{,\alpha}g_{ed,\beta}\tau^e_\gamma \\
&=& g^{da}(g_{ed,\alpha}\tau^e_{\gamma,\beta} - g_{ed,\beta}\tau^e_{\gamma,\alpha} ) - g^{ea}g^{dh}g_{eh,\beta}g_{kd,\alpha}\tau^k_\gamma \\
&& + g^{ea}g^{dh}g_{eh,\alpha}g_{kd,\beta}\tau^k_\gamma \, ,
\end{eqnarray*}
and therefore we have 
\begin{eqnarray} \label{aux2}
&&(g_{ab}-\sigma_{ab})\tau^b_\nu \left( (g^{da}g_{ed,\alpha}\tau^e_{\gamma})_{,\beta} - (g^{da}g_{ed,\beta}\tau^e_{\gamma})_{,\alpha}  \right) \nonumber \\
&=& (g_{ab}-\sigma_{ab})\tau^b_\nu \left(g^{da}(g_{ed,\alpha}\tau^e_{\gamma,\beta} - g_{ed,\beta}\tau^e_{\gamma,\alpha} ) - g^{ea}g^{dh}g_{eh,\beta}g_{kd,\alpha}\tau^k_\gamma + g^{ea}g^{dh}g_{eh,\alpha}g_{kd,\beta}\tau^k_\gamma\right) \nonumber \\
&=& g_{ab}\tau^b_\nu g^{da}(g_{ed,\alpha}\tau^e_{\gamma,\beta} - g_{ed,\beta}\tau^e_{\gamma,\alpha} ) -g_{ab}\tau^b_\nu g^{ea}g^{dh}g_{eh,\beta}g_{kd,\alpha}\tau^k_\gamma + g_{ab}\tau^b_\nu g^{ea}g^{dh}g_{eh,\alpha}g_{kd,\beta}\tau^k_\gamma \nonumber \\
&& -\sigma_{ab}\tau^b_\nu (g^{da}(g_{ed,\alpha}\tau^e_{\gamma,\beta} - g_{ed,\beta}\tau^e_{\gamma,\alpha} ) - g^{ea}g^{dh}g_{eh,\beta}g_{kd,\alpha}\tau^k_\gamma + g^{ea}g^{dh}g_{eh,\alpha}g_{kd,\beta}\tau^k_\gamma) \nonumber \\
&=& \tau^d_\nu (g_{ed,\alpha}\tau^e_{\gamma,\beta} - g_{ed,\beta}\tau^e_{\gamma,\alpha} ) - \tau^e_\nu g^{dh}g_{eh,\beta}g_{kd,\alpha}\tau^k_\gamma + \tau^e_\nu g^{dh}g_{eh,\alpha}g_{kd,\beta}\tau^k_\gamma \nonumber \\
&& -\sigma_{ab}\tau^b_\nu (g^{da}(g_{ed,\alpha}\tau^e_{\gamma,\beta} - g_{ed,\beta}\tau^e_{\gamma,\alpha} ) - g^{ea}g^{dh}g_{eh,\beta}g_{kd,\alpha}\tau^k_\gamma + g^{ea}g^{dh}g_{eh,\alpha}g_{kd,\beta}\tau^k_\gamma) \nonumber \\
&=& \tau^d_\nu (g_{ed,\alpha}\tau^e_{\gamma,\beta} - g_{ed,\beta}\tau^e_{\gamma,\alpha} ) -\sigma_{ab}\tau^b_\nu g^{da}g_{ed,\alpha}\tau^e_{\gamma,\beta} +\sigma_{ab}\tau^b_\nu g^{da} g_{ed,\beta}\tau^e_{\gamma,\alpha}  \nonumber \\
&& + \sigma_{ab}\tau^b_\nu g^{ea}g^{dh}g_{eh,\beta}g_{kd,\alpha}\tau^k_\gamma -\sigma_{ab}\tau^b_\nu g^{ea}g^{dh}g_{eh,\alpha}g_{kd,\beta}\tau^k_\gamma \nonumber \\
&\stackrel{\mbox{M2}}{=}& \tau^d_\nu (g_{ed,\alpha}\tau^e_{\gamma,\beta} - g_{ed,\beta}\tau^e_{\gamma,\alpha} ) -\frac{1}{2}\sigma_{ab}\tau^b_\nu \sigma^{da}(\sigma_{ed,\alpha}+g_{ed,\alpha})\tau^e_{\gamma,\beta} +\frac{1}{2}\sigma_{ab}\tau^b_\nu \sigma^{da}(\sigma_{ed,\beta}+g_{ed,\beta})\tau^e_{\gamma,\alpha}  \nonumber \\
&& + \frac{1}{4}\sigma_{ab}\tau^b_\nu \sigma^{ea}(\sigma_{eh,\beta}+g_{eh,\beta})\sigma^{dh}(\sigma_{kd,\alpha}+g_{kd,\alpha})\tau^k_\gamma \nonumber \\
&& -\frac{1}{4}\sigma_{ab}\tau^b_\nu \sigma^{ea}(\sigma_{eh,\alpha}+g_{eh,\alpha})\sigma^{dh}(\sigma_{kd,\beta}+g_{kd,\beta}) \tau^k_\gamma \nonumber \\
&=& \tau^d_\nu (g_{ed,\alpha}\tau^e_{\gamma,\beta} - g_{ed,\beta}\tau^e_{\gamma,\alpha} ) -\frac{1}{2}\tau^d_\nu (\sigma_{ed,\alpha}+g_{ed,\alpha})\tau^e_{\gamma,\beta} +\frac{1}{2}\tau^d_\nu (\sigma_{ed,\beta}+g_{ed,\beta})\tau^e_{\gamma,\alpha} \nonumber \\
&& + \frac{1}{4} \tau^e_\nu (\sigma_{eh,\beta}+g_{eh,\beta})\sigma^{dh}(\sigma_{kd,\alpha}+g_{kd,\alpha})\tau^k_\gamma -\frac{1}{4}\tau^e_\nu (\sigma_{eh,\alpha}+g_{eh,\alpha})\sigma^{dh}(\sigma_{kd,\beta}+g_{kd,\beta}) \tau^k_\gamma \nonumber \\
&=& \tau^d_\nu (g_{ed,\alpha}\tau^e_{\gamma,\beta} - g_{ed,\beta}\tau^e_{\gamma,\alpha} ) -\frac{1}{2}\tau^d_\nu (\sigma_{ed,\alpha}+g_{ed,\alpha})\tau^e_{\gamma,\beta} +\frac{1}{2}\tau^d_\nu (\sigma_{ed,\beta}+g_{ed,\beta})\tau^e_{\gamma,\alpha} \, .
\end{eqnarray}
We will also use
\begin{eqnarray} \label{aux3}
&& -2\sigma_{be}g^{ed}g_{ad,\beta}\tau^b_\nu g^{ha}g_{kh,\alpha}\tau^k_{\gamma} + 2\sigma_{be}g^{ed}g_{ad,\alpha}\tau^b_\nu g^{ha}g_{kh,\beta}\tau^k_{\gamma} \nonumber \\
&\stackrel{\mbox{M2}}{=}& -\frac{1}{2}\sigma_{be}\sigma^{ed}(\sigma_{ad,\beta}+g_{ad,\beta})\tau^b_\nu \sigma^{ha}(\sigma_{kh,\alpha}+g_{kh,\alpha})\tau^k_{\gamma}  \nonumber \\
&& + \frac{1}{2}\sigma_{be}\sigma^{ed}(\sigma_{ad,\alpha}+g_{ad,\alpha})\tau^b_\nu \sigma^{ha}(\sigma_{kh,\beta}+g_{kh,\beta})\tau^k_{\gamma} \nonumber \\
&=& -\frac{1}{2}(\sigma_{ab,\beta}+g_{ab,\beta})\tau^b_\nu \sigma^{ha}(\sigma_{kh,\alpha}+g_{kh,\alpha})\tau^k_{\gamma} + \frac{1}{2}(\sigma_{ab,\alpha}+g_{ab,\alpha})\tau^b_\nu \sigma^{ha}(\sigma_{kh,\beta}+g_{kh,\beta})\tau^k_{\gamma} =0 \, .
\end{eqnarray}
Now we will finally check that $R$ vanishes using M1, M2 and M3. Recall that in the computation below we omit the term $\dot{x}^\nu\dot{x}^\gamma$:
\begin{eqnarray*}
R&=& ((g_{ab}\tau^b_\nu+g_{\nu a})\tau^a_{\alpha,\gamma})_{,\beta} - ((g_{ab}\tau^b_\nu+g_{\nu a})\tau^a_{\beta,\gamma})_{,\alpha} 
\stackrel{\mbox{M1}}{=} ((g_{ab}-\sigma_{ab})\tau^b_\nu\tau^a_{\alpha,\gamma})_{,\beta} - ((g_{ab}-\sigma_{ab})\tau^b_\nu\tau^a_{\beta,\gamma})_{,\alpha}\\
&\stackrel{\mbox{M3}}{=}& ((g_{ab}-\sigma_{ab})\tau^b_\nu(\tau^a_{\gamma,\alpha}+g^{da}g_{ed,\alpha}\tau^e_{\gamma}))_{,\beta} - ((g_{ab}-\sigma_{ab})\tau^b_\nu(\tau^a_{\gamma,\beta}+g^{da}g_{ed,\beta}\tau^e_{\gamma}))_{,\alpha}\\
&=& (g_{ab}-\sigma_{ab})_{,\beta}\tau^b_\nu(\tau^a_{\gamma,\alpha}+g^{da}g_{ed,\alpha}\tau^e_{\gamma})
+(g_{ab}-\sigma_{ab})\tau^b_{\nu,\beta}(\tau^a_{\gamma,\alpha}+g^{da}g_{ed,\alpha}\tau^e_{\gamma}) \\
&&+(g_{ab}-\sigma_{ab})\tau^b_\nu(\tau^a_{\gamma,\alpha}+g^{da}g_{ed,\alpha}\tau^e_{\gamma})_{,\beta} 
 - (g_{ab}-\sigma_{ab})_{,\alpha}\tau^b_\nu(\tau^a_{\gamma,\beta}+g^{da}g_{ed,\beta}\tau^e_{\gamma}) \\
&& - (g_{ab}-\sigma_{ab})\tau^b_{\nu,\alpha}(\tau^a_{\gamma,\beta}+g^{da}g_{ed,\beta}\tau^e_{\gamma})
- (g_{ab}-\sigma_{ab})\tau^b_\nu(\tau^a_{\gamma,\beta}+g^{da}g_{ed,\beta}\tau^e_{\gamma})_{,\alpha} \\
&=& (g_{ab}-\sigma_{ab})_{,\beta}\tau^b_\nu(\tau^a_{\gamma,\alpha}+g^{da}g_{ed,\alpha}\tau^e_{\gamma})
+(g_{ab}-\sigma_{ab})\tau^b_{\nu,\beta}(g^{da}g_{ed,\alpha}\tau^e_{\gamma})
+(g_{ab}-\sigma_{ab})\tau^b_\nu(g^{da}g_{ed,\alpha}\tau^e_{\gamma})_{,\beta} \\
&& - (g_{ab}-\sigma_{ab})_{,\alpha}\tau^b_\nu(\tau^a_{\gamma,\beta}+g^{da}g_{ed,\beta}\tau^e_{\gamma})
- (g_{ab}-\sigma_{ab})\tau^b_{\nu,\alpha}(g^{da}g_{ed,\beta}\tau^e_{\gamma})
- (g_{ab}-\sigma_{ab})\tau^b_\nu(g^{da}g_{ed,\beta}\tau^e_{\gamma})_{,\alpha} \\
&\stackrel{(\ref{aux1}),(\ref{aux2})}{=}& 2g_{ab,\beta}\tau^b_\nu\tau^a_{\gamma,\alpha}-2\sigma_{be}g^{ed}g_{ad,\beta}\tau^b_\nu\tau^a_{\gamma,\alpha}
+ 2g_{ab,\beta}\tau^b_\nu g^{da}g_{ed,\alpha}\tau^e_{\gamma}-2\sigma_{be}g^{ed}g_{ad,\beta}\tau^b_\nu g^{da}g_{ed,\alpha}\tau^e_{\gamma} \\
&& -2g_{ab,\alpha}\tau^b_\nu\tau^a_{\gamma,\beta}+2\sigma_{be}g^{ed}g_{ad,\alpha}\tau^b_\nu\tau^a_{\gamma,\beta}
- 2g_{ab,\alpha}\tau^b_\nu g^{da}g_{ed,\beta}\tau^e_{\gamma}+2\sigma_{be}g^{ed}g_{ad,\alpha}\tau^b_\nu g^{da}g_{ed,\beta}\tau^e_{\gamma} \\
&& +\tau^d_{\nu,\beta}g_{ed,\alpha}\tau^e_{\gamma}-\sigma_{ab}\tau^b_{\nu,\beta}g^{da}g_{ed,\alpha}\tau^e_{\gamma} - \tau^d_{\nu,\alpha}g_{ed,\beta}\tau^e_{\gamma} + \sigma_{ab}\tau^b_{\nu,\alpha}g^{da}g_{ed,\beta}\tau^e_{\gamma} \\
&& +\tau^d_\nu (g_{ed,\alpha}\tau^e_{\gamma,\beta} - g_{ed,\beta}\tau^e_{\gamma,\alpha} ) -\frac{1}{2}\tau^d_\nu (\sigma_{ed,\alpha}+g_{ed,\alpha})\tau^e_{\gamma,\beta} +\frac{1}{2}\tau^d_\nu (\sigma_{ed,\beta}+g_{ed,\beta})\tau^e_{\gamma,\alpha} \\
&=& -2\sigma_{be}g^{ed}g_{ad,\beta}\tau^b_\nu\tau^a_{\gamma,\alpha}
+ 2g_{ab,\beta}\tau^b_\nu g^{da}g_{ed,\alpha}\tau^e_{\gamma}-2\sigma_{be}g^{ed}g_{ad,\beta}\tau^b_\nu g^{da}g_{ed,\alpha}\tau^e_{\gamma} \\
&& +2\sigma_{be}g^{ed}g_{ad,\alpha}\tau^b_\nu\tau^a_{\gamma,\beta}
- 2g_{ab,\alpha}\tau^b_\nu g^{da}g_{ed,\beta}\tau^e_{\gamma}+2\sigma_{be}g^{ed}g_{ad,\alpha}\tau^b_\nu g^{da}g_{ed,\beta}\tau^e_{\gamma} \\
&& -\sigma_{ab}\tau^b_{\nu,\beta}g^{da}g_{ed,\alpha}\tau^e_{\gamma} +\sigma_{ab}\tau^b_{\nu,\alpha}g^{da}g_{ed,\beta}\tau^e_{\gamma}  -\frac{1}{2}\tau^d_\nu (\sigma_{ed,\alpha}+g_{ed,\alpha})\tau^e_{\gamma,\beta} +\frac{1}{2}\tau^d_\nu (\sigma_{ed,\beta}+g_{ed,\beta})\tau^e_{\gamma,\alpha} \\
&=& -2\sigma_{be}g^{ed}g_{ad,\beta}\tau^b_\nu\tau^a_{\gamma,\alpha}
-2\sigma_{be}g^{ed}g_{ad,\beta}\tau^b_\nu g^{da}g_{ed,\alpha}\tau^e_{\gamma}  + 2\sigma_{be}g^{ed}g_{ad,\alpha}\tau^b_\nu\tau^a_{\gamma,\beta} + 2\sigma_{be}g^{ed}g_{ad,\alpha}\tau^b_\nu g^{da}g_{ed,\beta}\tau^e_{\gamma} \\
&& -\sigma_{ab}\tau^b_{\nu,\beta}g^{da}g_{ed,\alpha}\tau^e_{\gamma} + \sigma_{ab}\tau^b_{\nu,\alpha}g^{da}g_{ed,\beta}\tau^e_{\gamma}  -\frac{1}{2}\tau^d_\nu (\sigma_{ed,\alpha}+g_{ed,\alpha})\tau^e_{\gamma,\beta} +\frac{1}{2}\tau^d_\nu (\sigma_{ed,\beta}+g_{ed,\beta})\tau^e_{\gamma,\alpha} \\
&\stackrel{(\ref{aux3})}{=}& -2\sigma_{be}g^{ed}g_{ad,\beta}\tau^b_\nu\tau^a_{\gamma,\alpha} +2\sigma_{be}g^{ed}g_{ad,\alpha}\tau^b_\nu\tau^a_{\gamma,\beta}  -\sigma_{ab}\tau^b_{\nu,\beta}(g^{da}g_{ed,\alpha}\tau^e_{\gamma}) + \sigma_{ab}\tau^b_{\nu,\alpha}(g^{da}g_{ed,\beta}\tau^e_{\gamma}) \\
&&  -\frac{1}{2}\tau^d_\nu (\sigma_{ed,\alpha}+g_{ed,\alpha})\tau^e_{\gamma,\beta} +\frac{1}{2}\tau^d_\nu (\sigma_{ed,\beta}+g_{ed,\beta})\tau^e_{\gamma,\alpha} \\
&\stackrel{\mbox{M2}}{=}& -\sigma_{be}\sigma^{ed}(\sigma_{ad,\beta}+g_{ad,\beta})\tau^b_\nu\tau^a_{\gamma,\alpha} + \sigma_{be}\sigma^{ed}(\sigma_{ad,\alpha}+g_{ad,\alpha})\tau^b_\nu\tau^a_{\gamma,\beta} \\
&& -\frac{1}{2}\sigma_{ab}\tau^b_{\nu,\beta}\sigma^{da}(\sigma_{ed,\alpha}+g_{ed,\alpha})\tau^e_{\gamma} + \frac{1}{2}\sigma_{ab}\tau^b_{\nu,\alpha}\sigma^{da}(\sigma_{ed,\beta}+g_{ed,\beta})\tau^e_{\gamma} \\
&&  -\frac{1}{2}\tau^d_\nu (\sigma_{ed,\alpha}+g_{ed,\alpha})\tau^e_{\gamma,\beta} +\frac{1}{2}\tau^d_\nu (\sigma_{ed,\beta}+g_{ed,\beta})\tau^e_{\gamma,\alpha} \\
&=& -(\sigma_{ab,\beta}+g_{ab,\beta})\tau^b_\nu\tau^a_{\gamma,\alpha} + (\sigma_{ab,\alpha}+g_{ab,\alpha})\tau^b_\nu\tau^a_{\gamma,\beta}  -\frac{1}{2}\tau^b_{\nu,\beta}(\sigma_{eb,\alpha}+g_{eb,\alpha})\tau^e_{\gamma} + \frac{1}{2}\tau^b_{\nu,\alpha}(\sigma_{eb,\beta}+g_{eb,\beta})\tau^e_{\gamma} \\
&&  -\frac{1}{2}\tau^d_\nu (\sigma_{ed,\alpha}+g_{ed,\alpha})\tau^e_{\gamma,\beta} +\frac{1}{2}\tau^d_\nu (\sigma_{ed,\beta}+g_{ed,\beta})\tau^e_{\gamma,\alpha} =0 \, .\\
\end{eqnarray*}

\section*{Acknowledgments}

This work has been partially supported by research grants MTM2016-76702-P (MINECO) and the ICMAT Severo Ochoa project SEV-2015-0554 (MINECO). 
MFP has been financially supported by a FPU scholarship from MECD. AMB
has been supported in part by the National Science Foundation grant 
DMS-1613819 and AFOSR.

\bibliographystyle{plain}
\bibliography{References}

\end{document}